\documentclass[11pt,a4paper]{amsart}
\usepackage[T1]{fontenc}
\usepackage{textcomp}
\usepackage{lmodern}
\usepackage{bbm}
\usepackage{amssymb}

\usepackage{xspace}
\usepackage{enumitem}
\setenumerate[1]{label=(\alph{*}), ref=(\alph{*})}
\setenumerate[2]{label=(\roman{*}), ref=(\roman{*})}

\usepackage{xcolor}
\usepackage[colorlinks]{hyperref}
\definecolor{egraf}{rgb}{0.2,0.4,0}
\definecolor{oldv}{rgb}{0.7,0.7,0.7}

\theoremstyle{plain}
\newtheorem{thm}{Theorem}[section]
\newtheorem*{thm*}{Theorem}
\newtheorem*{lem*}{Lemma}
\newtheorem{cor}[thm]{Corollary}
\newtheorem{prop}[thm]{Proposition}
\newtheorem{lem}[thm]{Lemma}
\newtheorem{obs}[thm]{Observation}
\newtheorem{quest}[thm]{Question}

\theoremstyle{definition}
\newtheorem{rem}[thm]{Remark}

\newcommand{\RNP}{Radon--Nikod\'{y}m property\xspace}
\newcommand{\eps}{\varepsilon}

\newcommand{\Natural}{\mathbb N}
\newcommand{\Real}{\mathbb R}
\newcommand{\abs}[1]{\left\vert#1\right\vert}
\newcommand{\set}[1]{\left\{#1\right\}}
\newcommand{\cconv}{\mathop{\overline{\mathrm{conv}}}\nolimits}

\newcommand{\1}{\mathbbm{1}}

\newcommand{\norm}[1]{\left\Vert#1\right\Vert}

\DeclareMathOperator{\conv}{conv}
\DeclareMathOperator{\diam}{diam}
\DeclareMathOperator{\sgn}{sign}
\DeclareMathOperator{\supp}{supp}
\DeclareMathOperator{\cof}{cof}

\begin{document}
\title{Asymptotic geometry and delta-points}
\author[T.~A.~Abrahamsen]{Trond A.~Abrahamsen}
\address[T.~A.~Abrahamsen]{Department of Mathematics, University of
  Agder, Postboks 422, 4604 Kristiansand, Norway.}
\email{trond.a.abrahamsen@uia.no}
\urladdr{http://home.uia.no/trondaa/index.php3}

\author[V.~Lima]{Vegard Lima}
\address[V.~Lima]{Department of Engineering Sciences, University of Agder,
Postboks 509, 4898 Grimstad, Norway.}
\email{Vegard.Lima@uia.no}

\author[A.~Martiny]{Andr\'e Martiny}
\address[A.~Martiny]{Department of Mathematics, University of
  Agder, Postboks 422, 4604 Kristiansand, Norway.}
\email{andre.martiny@uia.no}

\author[Y.~Perreau]{Yoël Perreau}
\address[Y.~Perreau]{Laboratoire de Math\'ematiques de Besan\c con, Universit\'e Bourgogne Franche-Comt\'e, CNRS UMR-
6623, 16 route de Gray, 25030 Besan\c con C\'edex, Besan\c con, France}
\email{yoel.perreau@univ-fcomte.fr}



\thanks{
  This work was supported by the AURORA mobility programme (project
  number: 309597) from the Norwegian Research Council and the $``$PHC
  Aurora$"$ program (project number: 45391PF), funded by the French
  Ministry for Europe and Foreign Affairs, the French Ministry for
  Higher Education, Research and Innovation.
}

\maketitle

\begin{abstract}
  We study Daugavet- and $\Delta$-points in Banach spaces. A
  norm one element $x$ is a Daugavet-point (respectively a $\Delta$-point) if
  in every slice of the unit ball (respectively in every slice of the unit
  ball containing $x$) you can find another element of distance as
  close to $2$ from $x$ as desired. 
  
  In this paper we look for criteria and properties ensuring that
  a norm one element is not a Daugavet- or
  $\Delta$-point.
  We show that asymptotically uniformly smooth spaces
  and reflexive asymptotically uniformly convex spaces
  do not contain $\Delta$-points.
  We also show that the same conclusion holds true for the James tree space as well
  as for its predual.

  Finally we prove that there exists a superreflexive Banach space
  with a Daugavet- or $\Delta$-point provided there exists such a
  space satisfying a weaker condition. 

\end{abstract}

\section{Introduction}
\label{sec:introduction}

Daugavet- and $\Delta$-points first appeared in
\cite{AHLP} as natural pointwise versions of geometric characterizations of
the Daugavet property \cite[Lemma~2]{MR1784413} and of the so called
spaces with bad projections \cite[Theorem~1.4]{zbMATH02168839}
(also known as spaces with the diametral local diameter two
property (DLD2P) \cite{BGLPRZ-diametral}).
We refer to Section~\ref{sec:preliminaries} for precise definitions
and equivalent reformulations.

From their introduction on, Daugavet- and $\Delta$-points
attracted a lot of attention and were intensively studied in classical
Banach spaces (\cite{AHLP,ALM,ALMT,HPV}).  In particular a strong
emphasis was put on finding linear or geometric properties that would
prevent norm one elements in a space to be Daugavet- or
$\Delta$-points. It soon appeared that even nice properties which on a
global level prevent the space to have the Daugavet property or the
DLD2P do not provide an obstruction to the existence of Daugavet- or
$\Delta$-points in the space.  For example there exists a Banach space
with a 1-unconditional basis such that the set of Daugavet-points are
weakly dense in the unit ball \cite[Theorem~4.7]{ALMT}.

Another striking example illustrating this was obtained in the context
of Lipschitz-free spaces.  The study of Daugavet- and $\Delta$-points
in this context started in \cite{JRZ} where a characterization of
Daugavet-points in free-spaces over compact metric spaces was
discovered. It was observed that Daugavet-points have to be at
distance $2$ from every denting point of the unit ball in any given
Banach space, \cite[Proposition~3.1]{JRZ}, and it was proved that the
converse holds in every free-space in the compact setting. Extending
this result to the general setting, Veeorg was then able to provide in
\cite{Vee} a surprising example of a metric space whose free-space has
the \RNP (RNP) and admits a Daugavet-point.

In the present paper we continue the investigation of Daugavet- and
$\Delta$-points in general Banach spaces by focusing on the
interactions between those points and the asymptotic geometry of the
space. We provide new examples of Banach spaces failing to contain
$\Delta$-points and we introduce weaker notions which can be viewed as
a step forward in the direction of constructing an example of a
superreflexive space with a Daugavet- or a $\Delta$-point.

Let us now describe the content of the paper and expose our main
results. In Section~\ref{sec:preliminaries} we recall the notion of
slices, give the definition of Daugavet- and $\Delta$-points, and
state a few simple geometric lemmata. As a warm up we give a simple
proof that uniformly non-square spaces do not admit $\Delta$-points
and explain why simple considerations on the diameter of slices cannot
rule out $\Delta$-points outside of this setting. We end the section
with the necessary background on asymptotic uniform properties of
Banach spaces.

In Section~\ref{sec:asympt-unif-smoothn} we focus on asymptotic
smoothness. Our main result there is a condition on the modulus of
asymptotic smoothness $\overline{\rho}_X(t,x)$ at some point $x$ which
prevents the considered point to be a $\Delta$-point.
We will show in particular that no asymptotically smooth point can be
a $\Delta$-point. As a consequence we obtain that asymptotically uniformly smooth spaces
fail to contain $\Delta$-points. We then apply this theorem to obtain
new examples of classical spaces failing to contain $\Delta$-points,
specifically spaces with Kalton's property $(M^*)$ and the
predual $JT_*$ of the James tree space.

In Section~\ref{sec:asympt-near-unif} we obtain partial pointwise
results for asymptotic convexity. As a consequence we show that spaces
with the property $(\alpha)$ of Rolewicz and in particular reflexive
asymptotically uniformly convex spaces do not admit
$\Delta$-points. As a consequence we obtain that the Baernstein space $B$
as well as its dual do not admit $\Delta$-points.

Motivated by the result from the preceding section, we look in
Section~\ref{sec:james-tree-space} at the existence of
Daugavet- and $\Delta$-points in the James tree space $JT$. Our main
result there is that $JT$ does not admit $\Delta$-points.
We also have some partial results for the dual $JT^*$.

In the final section we introduce a natural weaker notion of $(2-\eps)$
Daugavet- and $\Delta$-points and we look at Banach spaces in which
such points exits for every $\eps>0$. We show that if there exists
a superreflexive space $X$ which has a $(2-\eps)$ Daugavet-point (respectively a $(2-\eps)$ $\Delta$-point)
for every $\eps > 0$, then there exists a superreflexive
space (an ultrapower of $X$) with a Daugavet-point (respectively a $\Delta$-point).

\section{Preliminaries}
\label{sec:preliminaries}

Let $X$ be a Banach space, $B_X$ its unit ball,
$S_X$ its unit sphere, and $X^*$ its dual space.
We consider real Banach spaces only.

For any $x^* \in S_{X^*}$ and $\delta > 0$ we define
a \emph{slice} of $B_X$ by
\begin{equation*}
  S(x^*,\delta)
  :=
  \{ y \in B_X : x^*(y) > 1 - \delta \}.
\end{equation*}
The corresponding closed slice is denoted
\begin{equation*}
  \overline{S}(x^*,\delta)
  :=
  \set{ y \in B_X : x^*(y) \ge 1 - \delta }.
\end{equation*}

Slices and closed slices of $B_X$ can also be defined using non-zero
functionals in $X^*$ by replacing the $1$ above by the norm of the
corresponding functional. Now if $x^*$ is a non-zero functional we
have $S(x^*,\delta) =
S\left(\frac{x^*}{\norm{x^*}}, \frac{\delta}{\norm{x^*}}\right)$
and we would like to point out that working with non-normalized
functionals can cause delicate computational problems
(see e.g. the proof of Theorem~\ref{thm:alpha_of_slice_not_delta}).
We will avoid them as much as possible.
We will also informally say that a Banach space has
\emph{small slices of arbitrary large diameter},
when the diameter of $S(x^*,\delta)$ is as close as we want
to $2$ for small $\delta$.

For every $x\in X$,
let us write $D(x):=\{ x^*\in S_{X^*} : \ x^*(x) = \norm{x} \}$.
We will make extensive use of the following lemma.

\begin{lem}\label{intersection_lemma}
  Let $x\in S_X$. For every $n\geq 1$, $\delta>0$,
  and $x_1^*, \ldots, x_n^*\in D(x)$, we have:
  \begin{equation*}
    S\left(
      \frac{1}{n}\sum_{i=1}^nx_i^*, \frac{\delta}{n}
    \right)
    \subset \bigcap_{i=1}^n S(x_i^*,\delta).
  \end{equation*}
\end{lem}

\begin{proof}
  Let $x^* = \frac{1}{n} \sum_{i=1}^n x_i^*$.
  Since every $x_i^*$ is in $D(x)$,
  we immediately obtain that $x^*$ is in $D(x)$
  and in particular that $\norm{x^*} = 1$.
  Now if $y \in S\left( x^*, \frac{\delta}{n} \right)$,
  we have, for every $1 \leq i_0 \leq n$,
  \begin{equation*}
    x_{i_0}^*(y) = nx^*(y) - \sum_{i=1\atop i\neq i_0}^n x_i^*(y)
    > n - \frac{n\delta}{n} - (n-1)
    = 1 - \delta,
  \end{equation*}
  and the conclusion follows.
\end{proof}

Note that the fact that every functional $x_i^*$ attains its norm at
$x$ is crucial here to guarantee that the average of the $x_i^*$
remains a norm one functional.

Let $X$ be a Banach space and $x \in S_X$.
For $\varepsilon > 0$ we define
\begin{equation*}
  \Delta_\varepsilon(x) :=
  \{
  y \in B_X : \|x - y\| \ge 2 - \varepsilon
  \}.
\end{equation*}
Following \cite{AHLP}
we say that $x$ is a \emph{Daugavet-point} if
$B_X = \overline{\conv} \Delta_\varepsilon(x)$
for all $\varepsilon > 0$
and we say that $x$ is a \emph{$\Delta$-point} if
$x \in \overline{\conv} \Delta_\varepsilon(x)$
for all $\varepsilon > 0$.

Using Hahn--Banach separation we get the following
well known lemma that we will use without reference.

\begin{lem}
  Let $X$ be a Banach space
  and let $x \in S_X$.
  \begin{enumerate}
  \item
$x$ is a Daugavet-point
    if and only if
    for all $\varepsilon > 0$, all $\delta > 0$,
    and all $x^* \in S_{X^*}$,
    there exists $y \in S(x^*,\delta)$ such that
    $\|x - y\| > 2 - \varepsilon$.
  \item
  $x$ is a $\Delta$-point
    if and only if
    for all $\varepsilon > 0$, all $\delta > 0$,
    and all $x^* \in S_{X^*}$
    such that $x \in S(x^*,\delta)$, there exists
    $y \in S(x^*,\delta)$ such that
    $\|x - y\| > 2 - \varepsilon$.
  \end{enumerate}
\end{lem}

The above lemma appears in \cite[Lemma~2.1]{AHLP},
but note that there is a misprint in the statement
of (2) in \cite[Lemma~2.1]{AHLP}. Following this we say that an
element $x^*\in S_{X^*}$ is a
\emph{weak$^*$ Daugavet-point} if for all $\varepsilon > 0$, all
$\delta > 0$, and all $x \in S_X$ (naturally
identified with an element of $X^{**}$) there exists $y^* \in S(x,\delta)$
such that $\|x^* - y^*\| > 2 - \varepsilon$ , and we say that $x^*\in S_{X^*}$ is a \emph{weak$^*$ $\Delta$-point} if for all
$\varepsilon > 0$, all $\delta > 0$, and all $x \in S_X$ such that $x^* \in
S(x,\delta)$, there exists $y^* \in S(x,\delta)$ such that $\|x^* -
y^*\| > 2 - \varepsilon$.

A Banach space is said to be \emph{uniformly non-square}
if there exists $\varepsilon > 0$ such that
for all $x,y \in B_X$ we have either
$\|\frac{1}{2}(x + y)\| \le 1 - \varepsilon$
or $\|\frac{1}{2}(x - y)\| \le 1 - \varepsilon$.

Uniformly non-square spaces were introduced by James~\cite{James-nsq}.
It is well known that uniformly non-square spaces
are superreflexive and that uniformly convex and uniformly
smooth spaces are uniformly non-square
(see e.g. \cite[Proposition~1]{MR1829721}).
The following simple observations characterizes
uniformly non-square spaces in terms of slices.

\begin{prop}\label{prop:unif_nsq_char}
  Let $X$ be a Banach space.
  The following are equivalent.
  \begin{enumerate}
  \item\label{item:unsq-char-1}
    $X$ is uniformly non-square.
  \item\label{item:unsq-char-2}
    There exists $\varepsilon > 0$
    such that for all $x^* \in S_{X^*}$
    the diameter of $S(x^*,\varepsilon)$
    is less than $2 - 2\varepsilon$.
  \end{enumerate}
\end{prop}

\begin{proof}
  \ref{item:unsq-char-1} $\Rightarrow$ \ref{item:unsq-char-2}.
  Let $\varepsilon > 0$ be such that
  for all $x,y \in B_X$ we have either
  $\|\frac{1}{2}(x + y)\| \le 1 - \varepsilon$
  or $\|\frac{1}{2}(x - y)\| \le 1 - \varepsilon$.
  If $x^* \in S_{X^*}$ and $x, y \in S(x^*,\varepsilon)$,
  then $\|x + y\| \ge x^*(x + y) \ge 2 - 2 \varepsilon$,
  so $\|x - y\| \le 2 - 2 \varepsilon$.

  \ref{item:unsq-char-2} $\Rightarrow$ \ref{item:unsq-char-1}.
  If $X$ is not uniformly non-square, then
  for every $\varepsilon > 0$ there exists
  $x,y \in B_X$ such that
  $\|x + y\| > 2 - \varepsilon$ and
  $\|x - y\| > 2 - \varepsilon$.

  Find $x^* \in S_{X^*}$ such that
  $x^*(x + y) > 2 - \varepsilon$
  and observe that $x,y \in S(x^*,\varepsilon)$
  with $\|x - y\| > 2 - \varepsilon$.
\end{proof}

If $x$ is a $\Delta$-point then
for all $\delta > 0$ and $x^* \in S_{X^*}$
with $x^*(x) > 1 - \delta$ we have
$\diam S(x^*,\delta) = 2$
so we immediately get the following.

\begin{cor}\label{cor:unif-nsq-no-Delta}
  Let $X$ be a uniformly non-square Banach space.
  Then $X$ does not admit $\Delta$-points.
\end{cor}

If a Banach space $X$ is \emph{not} uniformly non-square
then for every $\varepsilon > 0$
there is a slice $S(x^*,\varepsilon)$, $x^* \in S_{X^*}$,
with diameter strictly greater than $2-2\varepsilon$, so $X$ admits small slices of arbitrarily large diameter.
Observe that if $X$ is a dual space, then
we may assume that the functional $x^*$
in the proof of \ref{item:unsq-char-2} $\Rightarrow$
\ref{item:unsq-char-1} in Proposition~\ref{prop:unif_nsq_char}
is weak$^*$ continuous.
Hence a dual space which is not uniformly non-square
has a unit ball with small weak$^*$ slices of diameter
arbitrary close to $2$.
These simple observations show that for spaces that are
not uniformly non-square there is no hope of ruling out
$\Delta$-points by some upper bound on the diameter of
slices of the unit ball.
Note that any Banach space $X$ of dimension greater or equal to
$2$ has an equivalent norm $|\cdot|$ such that $(X,|\cdot|)$
is \emph{not} uniformly non-square \cite[Corollary~1]{MR1829721}.
In particular, Corollary~\ref{cor:unif-nsq-no-Delta} does
not rule out $\Delta$-points in superreflexive or
finite dimensional spaces.

Let $X$ be a Banach space.
Let $\cof(X)$ denote the set of all subspaces of finite co-dimension of $X$.
For $t > 0$ and $x \in S_X$ consider
\begin{equation*}
  \bar{\rho}_X(t,x)
  =
  \inf_{Y \in \cof(X)} \sup_{y \in S_Y} \{ \|x + ty\| - 1 \}
\end{equation*}
and
\begin{equation*}
  \bar{\delta}_X(t,x)
  =
  \sup_{Y \in \cof(X)} \inf_{y \in S_Y} \{ \|x + ty\| - 1 \}.
\end{equation*}
The \emph{modulus of asymptotic convexity} of $X$
is given by
\begin{equation*}
  \bar{\delta}_X(t) = \inf_{x \in S_X} \bar{\delta}_X(t,x).
\end{equation*}
and the \emph{modulus of asymptotic smoothness} of $X$
is given by
\begin{equation*}
  \bar{\rho}_X(t) = \sup_{x \in S_X} \bar{\rho}_X(t,x).
\end{equation*}
The space $X$ is said to be \emph{asymptotically uniformly smooth}
(AUS for short) if $\lim_{t \to 0} t^{-1}\bar{\rho}_X(t) = 0$
and it is \emph{asymptotically uniformly convex} (AUC)
if $\bar{\delta}_X(t) > 0$ for all $t > 0$.
Similarly in $X^*$ there is a weak$^*$-modulus
of asymptotic uniform convexity defined by
\begin{equation*}
  \bar{\delta}^*_{X}(t) = \inf_{x^* \in S_{X^*}}
  \sup_{E} \inf_{y^* \in S_E} \{ \|x^* + ty^*\| - 1 \},
\end{equation*}
where $E$ runs through all weak$^*$-closed subspaces
of $X^*$ of finite codimension. We say that $X^*$
is \emph{AUC}$^*$ if $\bar{\delta}^*_{X}(t) > 0$
for all $t > 0$. Clearly, a dual space which is
AUC$^*$ is also AUC. The converse does not hold true since the James tree space $JT$ is known to have an AUC dual space (see \cite{Gir}) but to admit no equivalent norm whose dual norm is AUC$^*$.

It is well known (see e.g. \cite[Corollay~2.4]{MR3596040})
that $X$ is AUS if and only if $X^*$ is AUC$^*$,
and that if $X$ is reflexive then
$X$ is AUS if and only if $X^*$ is AUC.
The space $\ell_1$ is an example of a non-reflexive AUC space.

A Banach space $X$ is said to have the \emph{uniform Kadec--Klee} (UKK)
property if for every $\varepsilon > 0$, there exists $\delta > 0$
such that whenever $(x_n) \subset B_X$
with $\|x_n - x_m\| \ge \varepsilon$ for all $n \neq m$
and $x_n \to x$ weakly for some $x \in X$, then it follows
that $\|x\| \le 1 - \delta$. This property was introduced
by Huff \cite{MR595102}.
The UKK property is a close relative of asymptotic uniform convexity.
It is not too difficult too show that $X$ is AUC if and only if
for all $\varepsilon > 0$ there exists $\delta > 0$ such that
if $x \in B_{X}$ and all weak neighborhoods $V$ of $x$
satisfy $\diam(V \cap B_{X}) > \varepsilon$, then $\|x\| \le 1- \delta$.
The latter property is a generalization of UKK and it shows
$X$ AUC implies $X$ UKK. (This generalization was used
in e.g. \cite{MR1333523}, and Lancien actually used it
as a definition of UKK.)
It is known that if $X$ does not contain a copy of $\ell_1$,
then the reverse implication holds and $X$ is AUC
if and only if $X$ is UKK. In particular,
AUC and UKK are equivalent for reflexive spaces. The idea is that  a
Banach space $X$ failing this property contains a point $x\in S_X$ and
a net $(x_V) \subset B_X$ converging weakly to $x$ such that
$\norm{x-x_V} \geq \eps$ for every weak neighborhood $V$ of $x$ for
some fixed $\eps > 0$.
Now if $X$ does not contain $\ell_1$, the non-separable version of
Rosenthal's result \cite[Theorem~3]{Ros} proved in
\cite[Theorem~2.6]{GG} tells us that the set $\{x_V\}$ is weakly
sequentially dense in its weak-closure and we can thus extract a
sequence converging weakly to $x$ and providing an obstruction to the
UKK property (up to further extractions in order to obtain an
$\frac{\eps}{2}$-separated sequence).

Let $X$ be a Banach space and $A$ a bounded subset of $X$.
By $\alpha(A)$ we denote the \emph{Kuratowski measure of non-compactness}
of $A$ which is defined as the infimum of $\varepsilon > 0$ such that
$A$ can be covered by a finite number of sets with
diameters less than $\varepsilon$, that is,
\begin{equation*}
  \alpha(A) := \inf \set{
    \varepsilon > 0 :
    A \subset \bigcup_{i=1}^n A_i, A_i \subset X,
    \diam(A_i) < \varepsilon, i = 1,2,\ldots,n
  }.
\end{equation*}
A Banach space $X$ has \emph{Rolewicz' property $(\alpha)$}
if for every $x^* \in S_{X^*}$ and $\varepsilon > 0$
there exists $\delta > 0$ such that
$\alpha(S(x^*,\delta)) \le \varepsilon$.
We say that $X$ has \emph{uniform property $(\alpha)$}
if the same $\delta$ works for all $x^* \in S_{X^*}$.
These properties were introduced by Rolewicz in \cite{MR928575}.
If $X$ has Rolewicz' property $(\alpha)$, then it is reflexive
(see e.g. \cite{MR924764}).

Implicit in Rolewicz \cite[Theorem~3]{MR928575}
is the result that $X$ is AUC and reflexive
if and only if $X$ has uniform property $(\alpha)$,
Rolewicz uses the term ``$X$ is $\Delta$-uniformly
convex'' instead of $X$ is AUC and reflexive.
It is known $X$ is AUC and reflexive
if and only if $X$ is nearly uniformly convex (NUC) \cite{MR595102}
if and only if $X$ is $\Delta$-uniformly convex.
The difference between these two types of uniform convexity
is that they use different (but equivalent) measures of
non-compactness.

We also note that for dual spaces we have that
if $X^*$ is AUC$^*$ then $X^*$ is has \emph{weak$^*$}
uniform property $(\alpha)$, that is,
for all $\varepsilon > 0$ there exists $\delta > 0$
such that $\alpha(S(x,\delta)) \le \varepsilon$
for all $x \in S_X$.
Our Corollary~\ref{cor:local_aus_kuc}
is a pointwise version of this result.
Lennard \cite[Proposition~1.3]{MR943795}
states that $X^*$ has the weak$^*$ version
of the uniform Kadec--Klee property if and only
if $X^*$ has weak$^*$ uniform property $(\alpha)$.
Note that Lennard credits this result to Sims and
he uses a different measure of non-compactness.

\section{Asymptotic uniform smoothness}
\label{sec:asympt-unif-smoothn}

The goal of this section is to show that Banach
spaces with an asymptotic uniformly smooth norm
do not admit $\Delta$-points.

Our first result connects the pointwise
modulus of asymptotic smoothness $\overline{\rho}_X(t,x)$ at $x \in S_X$
with the measure of non-compactness of the slices
defined by $x$.
In his thesis Dutrieux gave a proof that a separable Banach space is
AUS if and only if its dual is weak$^*$ uniformly
Kadec--Klee \cite[Proposition~36]{DutThesis}.
Our proof of the following proposition follows closely
part of the proof given by Dutrieux,
but we do not assume separability.

\begin{prop}\label{prop:optimized_aus_kuc}
  Let $X$ be a Banach space, and fix $x\in S_X$ and $\eps>0$.
 If there exists $t>0$ such that $\bar{\rho}_X(t,x)/t < \varepsilon$, then we can find $\delta>0$ such that $\alpha (S(x,\delta))< 2\varepsilon$.

\end{prop}

To prove this proposition we will need two well known lemmas.
The first lemma is from \cite[Lemma~2.13]{MR1888429}.
A nice different proof can be found in the thesis
of Dutrieux \cite[Lemma~38]{DutThesis}.

\begin{lem}\label{lem:Lemme38}
  Let $X$ be a Banach space.
  For all $Y \in \cof(X)$ and all $\varepsilon > 0$,
  there exists a compact set $K_\varepsilon$ such that
  $B_X \subset K_\varepsilon + (2+\varepsilon)B_Y$.
\end{lem}

The proof of the next lemma is simple using contradiction,
so we skip it.

\begin{lem}\label{lem:net_diam_limsup}
  Let $X$ be a Banach space.
  Let $x^* \in B_{X^*}$ and $\varepsilon > 0$.

  If $\limsup_\alpha \|x^* - x^*_\alpha\| < \varepsilon$
  whenever $(x^*_\alpha) \subseteq B_{X^*}$ and
  $x^*_\alpha \overset{w^*}{\to} x^*$,
  then there exists a weak$^*$-neighborhood $V$
  of $x^*$ with $V \cap B_{X^*} \subseteq B(x^*,\varepsilon)$.
\end{lem}

\begin{proof}[Proof of Proposition~\ref{prop:optimized_aus_kuc}]
  Let us assume that $\bar{\rho}_X(t,x)/t < \varepsilon$ and let us write $\delta_x=\eps t-\bar{\rho}_X(t,x)$. We then have the following.  
  \bigbreak

  \textbf{Claim.}
  Take any $\delta<\delta_x$.
  For every $x^* \in S(x,\delta)$ and every
  net $(x_\alpha^*) \subset B_{X^*}$ such that
  $x_\alpha^* \to x^*$ weak$^*$ we have
  $\limsup_\alpha \|x^* - x^*_\alpha\| < 2\varepsilon$.

\bigbreak

  Before proving the claim let us see how to finish the proof. 
  Pick any $\delta<\delta_x$ and take $\delta'$ such that $\delta<\delta'<\delta_x$.
  By the claim and  by 
  Lemma~\ref{lem:net_diam_limsup} we can find for every $x^* \in S(x,\delta')$
  a weak$^*$-neighborhood $V_{x^*}$ of $x^*$ with
  $V_{x^*} \cap B_{X^*} \subseteq B(x^*,2\varepsilon)$.
  Then $(V_{x^*})_{x^* \in S(x,\delta')}$ is an
  open cover of $S(x,\delta')$ and therefore
  an open cover of the weak$^*$ compact set
  $\bar{S}(x,\delta)$. By compactness there is a finite
  subcover.
  
  \bigbreak

  \textit{Proof of the claim.}
  Fix some $\delta<\delta_x$ and let $x^* \in S(x,\delta)$ and
  $(x_\alpha^*) \subset B_{X^*}$ such that
  $x_\alpha^* \to x^*$ weak$^*$.
  We will show that $L=\limsup_\alpha \|x^* - x^*_\alpha\|
   < 2\varepsilon$.
  Again pick any $\delta'$ such that $\delta<\delta'<\delta_x$. By
  definition of $\delta_x$ we then have $\bar{\rho}_X(t,x) < \eps
  t-\delta'$
  so there exists $Z \in \cof(X)$ such that
  \begin{equation*}
    \sup_{z \in S_Z} \|x + tz\| \le 1 + \eps t -\delta'.
  \end{equation*}
  Since for any $s \in [0,1]$
  \begin{equation*}
    \|x + tsz\|
    =
    \|(1-s)x + s(x + tz)\|
    \le
    1 + \eps t -\delta',
  \end{equation*}
  we get
  \begin{equation*}
    \sup_{z \in B_Z} \|x + tz\| \le 1 + \eps t -\delta'.
  \end{equation*}
  
  Now let $Y = Z \cap \ker(x^*)$ and let $\eta > 0$.
  By Lemma~\ref{lem:Lemme38} there exists a compact set $K$
  in $X$ such that $B_X \subseteq K + (2+\eta)B_Y$.

  By compactness of $K$ and boundedness of $(x^*_\alpha)$
  we have that $x^*_\alpha \to x^*$ uniformly on $K$.
  We may therefore choose $\beta$ such that
  \begin{itemize}
  \item
    $|\langle x^*_\beta - x^*, k \rangle| < \eta$
    for all $k \in K$;
  \item
    $x^*_\beta(x) > 1 - \delta$;
  \item
    $|\|x^*_\beta - x^*\| - L| < \eta$.
  \end{itemize}

  Choose $x_\beta \in S_X$ such that
  $\langle x^*_\beta - x^*, x_\beta \rangle > L - \eta$.
  Now write $x_\beta = k_\beta + (2 + \eta)y_\beta$
  with $k_\beta \in K$ and $y_\beta \in B_Y$.
  We get
  \begin{equation*}
    x^*_\beta (y_\beta)
    =
    \langle x^*_\beta - x^*, y_\beta \rangle
    =
    \frac{\langle x^*_\beta - x^*, x_\beta \rangle
      - \langle x^*_\beta - x^*, k_\beta \rangle}
    {2 + \eta}
    > \frac{L - 2\eta}{2 + \eta}
  \end{equation*}
  since $y_\beta \in \ker(x^*)$.
  Therefore
  \begin{equation*}
    1 - \delta
    +
    \frac{t(L-2\eta)}{2+\eta}
    <
    \langle x_\beta^*, x \rangle
    +
    \frac{t(L-2\eta)}{2+\eta}
    <
    \langle x^*_\beta, x + t y_\beta \rangle
    \le
    \|x + ty_\beta\|
    \le 1 + \eps t -\delta'.
  \end{equation*}
  Finally 
  \begin{equation*}
    L - 2\eta < (2+\eta)(\eps-\theta)
  \end{equation*} with $\theta=\frac{\delta'-\delta}{t}>0$.
 Since $\eta > 0$ was arbitrary we get
  $L \le2(\eps-\theta)<2\eps$ as desired.
\end{proof}

As an immediate corollary we get.

\begin{cor}\label{cor:local_aus_kuc}

Let $X$ be a Banach space and let $x\in S_X$ be an asymptotically
smooth point, that is a point for which $\lim_{t\to 0}
\frac{\overline{\rho}_X(t,x)}{t} =0$. Then $\lim_{\delta \to
  0} \alpha(S(x,\delta))=0$.

\end{cor}

We will now show that this condition on the Kuratowski index of the slices
$S(x,\delta)$ prevents the point $x$ to be a $\Delta$-point.

\begin{thm}\label{thm:alpha_of_slice_not_delta}
  Let $X$ be a Banach space and let $x \in S_X$.
  If there exists a $\delta > 0$
  such that $\alpha(S(x,\delta)) < \frac{2}{3}$,
  then $x$ is not a $\Delta$-point.
\end{thm}

\begin{proof}
  Find $\delta > 0$ and $\varepsilon > 0$ such that
  $\alpha\left(S(x,\delta)\right) < \varepsilon < \frac{2}{3}$. Since
  for every $\delta'\leq \delta$, we have $S(x,\delta')\subset
  S(x,\delta)$ so $\alpha\left(S(x,\delta')\right) \leq
  \alpha\left(S(x,\delta)\right) < \varepsilon$, we may, and do, 
  assume that $\delta \leq  \varepsilon$.

  By assumption we can find $x_1^*,\ldots, x_n^*\in B_{X^*}$ with
  $n \geq 1$ such that
  $S(x,\delta)\subset \bigcup_{i=1}^n B(x_i^*,\varepsilon)$.
  In fact we have the following.

  \bigbreak
  \textbf{Claim 1.}
  For every $\delta' \leq \delta$,
  we can find $1 \leq m \leq n$ and $y_1^*, \ldots, y_m^* \in S_{X^*}$
  such that
  \begin{enumerate}
  \item
    $y_j^*\in S(x,\delta')$ for every $1\leq j\leq m$;
  \item
    $S_{X^*} \cap S(x,\delta')
    \subset
    \bigcup_{j=1}^m B(y_j^*, 2\varepsilon)$.
  \end{enumerate}

  \begin{proof}[Proof of Claim~1]
    Let us take $\delta' \leq \delta$.
    As observed before, we have
    $S(x, \delta') \subset S(x, \delta)$, and thus
    $S_{X^*} \cap S(x, \delta') \subset S(x, \delta)
    \subset \bigcup_{i=1}^n B(x_i^*, \varepsilon)$.

    Clearly the set $S_{X^*}\cap S(x, \delta')$ is not empty since it
    contains $D(x)$, so the set
    $J = \{ 1 \leq i \leq n: \ (S_{X^*} \cap S(x, \delta') )
    \cap B(x_i^*, \varepsilon) \neq \emptyset \}$
    has cardinality $\abs{J} = m$ with $1 \leq m \leq n$.

    Since obviously $S_{X^*} \cap S(x, \delta') \subset
    \bigcup_{j \in J} B(x_j^*, \varepsilon)$,
    the conclusion follows by picking, for every
    $j \in J$, an element $y_j^* \in
    (S_{X^*} \cap S(x, \delta') ) \cap B(x_j^*, \varepsilon)$,
    and by observing that
    $B(x_j^*, \varepsilon) \subset B(y_j^*, 2\varepsilon)$.
  \end{proof}
  \bigbreak

  Now let us find $1 \leq m \leq n$ and $y_1^*,\ldots y_m^*\in S_{X^*}$
  satisfying the properties of Claim~$1$ for
  $\delta' = \frac{\delta}{2n}$, and let us define
  $y^* = \frac{1}{m} \sum_{j=1}^m y_j^*$. We have the following.

  \bigbreak
  \textbf{Claim 2.}
  The slice
  $S(y^*, \delta') = \{y \in B_X : \ y^*(y) > \norm{y^*} -\delta' \}$
  satisfies
  \begin{equation*}
    x \in S(y^*, \delta') \subset \bigcap_{j=1}^m S(y_j^*, \delta).
  \end{equation*}

  \begin{proof}[Proof of Claim~2]
    Since every $y_j^*$ is in $S_{X^*}$ and satisfies
    $y_j^*(x) > 1-\delta'$, we clearly have $\norm{y^*}\leq 1$ and
    $y^*(x) > 1 - \delta' \geq \norm{y^*} - \delta'$ so
    $x \in S(y^*, \delta')$ and $\norm{y^*} \geq 1 - \delta'$.

    Now for every $y\in S(y^*,\delta')$ we have
    $y^*(y) > \norm{y^*} - \delta' \geq 1 - 2\delta'$
    so
    \begin{equation*}
      \frac{1}{m} y_k^*(y)
      > 1 - 2\delta' - \frac{1}{m} \sum_{j \neq k} y_j^*(y)
    \end{equation*}
    for every $1 \leq k \leq m$ and thus
    $y_k^*(y) > 1 - 2m \delta'
    = 1 - \frac{m\delta}{n} \geq 1 - \delta$
    since $m \leq n$.
  \end{proof}

  \bigbreak
  Let $y \in S(y^*,\delta')$ and $z^* \in S_{X^*}$.
  If $z^*(x) > 1 - \delta'$, then
  \begin{equation*}
    z^* \in S_{X^*} \cap S(x,\delta')
    \subset \bigcup_{j=1}^m B(y_j^*,2\varepsilon)
  \end{equation*}
  so there exists $y^*_{j_0}$ with
  $\|z^* - y_{j_0}^*\| \le 2\varepsilon$.
  Then
  \begin{equation*}
    z^*(y) \ge y_{j_0}^*(y) - 2\varepsilon
    > 1 - \delta - 2\varepsilon
  \end{equation*}
  and hence
  \begin{equation*}
    z^*(x - y) \le 1 - z^*(y)
    < \delta + 2\varepsilon \le 3\varepsilon.
  \end{equation*}
  While if $z^*(x) \le 1 - \delta'$, then
  \begin{equation*}
    z^*(x - y) \le 1 - \delta' + 1
    = 2 - \delta'.
  \end{equation*}
  Hence $\|x - y\| \le \max(2-\delta',3\varepsilon)$.
  We have $\delta' > 0$ and $\varepsilon <\frac{2}{3}$ so
  $x$ cannot be a $\Delta$-point.
\end{proof}

Combining Theorem~\ref{thm:alpha_of_slice_not_delta} and
Proposition~\ref{prop:optimized_aus_kuc} we then get

\begin{prop}

Let $X$ be a Banach space and let $x\in S_X$ satisfy $\overline{\rho}_X(t,x)<\frac{t}{3}$ for some $t>0$. Then $x$ is not a $\Delta$-point.

\end{prop}

In particular no asymptotically smooth point can be a $\Delta$-point and we obtain

\begin{thm}\label{thm:AUS_no_delta_points}
  Let $X$ be an AUS Banach space.
  Then $X$ does not admit a $\Delta$-point.
\end{thm}

Next let us collect some examples where the
above corollary applies.

Recall that a Banach space has Kalton's property $(M)$
if whenever $x,y \in X$ with $\|x\|=\|y\|$ and
$(x_\alpha)$ is a bounded weakly null net in $X$, then
\begin{equation*}
  \limsup_\alpha \| x + x_\alpha\|
  =
  \limsup_\alpha \| y + x_\alpha\|.
\end{equation*}
Similarly $X$ has property $(M^*)$
if whenever $x^*,y^* \in X^*$ with $\|x^*\|=\|y^*\|$ and
$(x^*_\alpha)$ is a bounded weak$^*$ null net in $X^*$, then
\begin{equation*}
  \limsup_\alpha \| x^* + x^*_\alpha\|
  =
  \limsup_\alpha \| y^* + x^*_\alpha\|.
\end{equation*}
If $X$ has property $(M^*)$, then $X$ has property
$(M)$ and $X$ is an $M$-ideal in $X^{**}$
(see e.g. \cite[Proposition~VI.4.15]{HWW}).
In particular, $X$ is an Asplund space
(see e.g. \cite[Theorem~III.3.1]{HWW}).
It is well known that property $(M^*)$ is
inherited by both subspaces and quotients
(see e.g. \cite{MR2000j:46034}).

By chasing references we find that the following
proposition holds.

\begin{prop}\label{prop:propM-AUS-char}
  Assume a Banach space $X$ has property $(M)$.
  The following are equivalent:
  \begin{enumerate}
  \item\label{item:pM1}
    $X$ is AUS;
  \item\label{item:pM2}
    $X$ contains no copy of $\ell_1$;
  \item\label{item:pM3}
    $X$ has property $(M^*)$.
  \end{enumerate}
\end{prop}

\begin{proof}
  \ref{item:pM1} $\Rightarrow$ \ref{item:pM2}.
  If $X$ is AUS, then $X$ is Asplund
  (see e.g. \cite[Proposition~2.4]{MR1888429}).
  Hence $X$ contains no copy of $\ell_1$.

  \ref{item:pM2} $\Rightarrow$ \ref{item:pM3}.
  If $X$ contains no copy of $\ell_1$,
  then no separable subspace of $X$ can contain $\ell_1$.
  Clearly every separable closed subspace of $X$
  has property $(M)$ (both net and sequential version,
  see \cite[Proposition~1]{Oja4})
  and then they all have property $(M^*)$
  (both net and sequential version)
  by Theorem~2.6 in \cite{KalW2}.
  Finally $X$ has property $(M^*)$ if every separable
  closed subspace does \cite[Proposition~3.1]{MR2000j:46034}.

  \ref{item:pM3} $\Rightarrow$ \ref{item:pM1}.
  Dutta and Godard \cite{MR2427395} proved that
  if $X$ is a separable Banach space with
  property $(M^*)$, then $X$ is AUS.
  However, using property ($M^*$) and
  Proposition~2.2 in \cite{MR3742968}
  one finds $\bar{\rho}_X(t,x) = \bar{\rho}_X(t)$
  for all $x \in S_X$ and their proof
  also works in the non-separable case.
\end{proof}

Since $c_0$ has property $(M^*)$ we have
that all subspaces and quotients of $c_0$
are AUS and they all fail to contain $\Delta$-points.
All these examples are $M$-ideals in their bidual,
that is, they are $M$-embedded \cite[Chapter~3]{HWW}.

Note that there are $M$-embedded spaces which are not AUS.
For example the Schreier space $\mathcal{S}$ is not AUS since
it does not have property $(M^*)$.
Indeed, if a Banach space $X$ has property $(M^*)$,
then the relative norm and weak$^*$ topologies
on $S_{X^*}$ coincide (see e.g. \cite[Proposition~VI.4.15]{HWW}).
But if $(e_i)$ is the unit vector basis in $\mathcal{S}$
and $(e_i^*)$ the the biorthogonal functionals in the dual,
then $e_2^* + e_i^* \in S_{\mathcal{S}^*}$ and converges weak$^*$ to
$e_2^*$, but not in norm. Note however that $\mathcal{S}$
does admit $\Delta$-point by Proposition~2.15 in \cite{ALMT}.

Let $X$ be a Banach space with a normalized basis $(e_i)$ (or more
generally an FDD $(E_i)$). We say that
\emph{$(e_i)$ admits block upper $\ell_p$ estimates} for some $p\in
(1,\infty) $ if there is a constant $C>0$ such that for every finite
blocks $x_1,\ldots, x_N$ of $(e_i)$ with consecutive disjoint supports
we have $\norm{\sum_{n=1}^Nx_n}^p\leq C\sum_{n=1}^N\norm{x_n}^p$. We
say that $(e_i)$ admits block lower $\ell_q$ estimates for some $q\in
(1,\infty) $ if there is a constant $c>0$ such that for every finite
blocks $x_1,\ldots, x_N$ of $(e_i)$ with consecutive disjoint supports
we have $\norm{\sum_{n=1}^Nx_n}^q\geq c\sum_{n=1}^N\norm{x_n}^q$. It
is well known that a basis admitting upper $\ell_p$ estimates is
shrinking while a basis admitting lower $\ell_q$ estimates is
boundedly complete. The latter can be proved by using the following
criterion, which is left as an exercise in \cite[Exercise~3.8]{AK} and
whose proof can be found in
\cite[Proposition~3.1]{CauLectureNotes}:
a basis is boundedly complete if and only if
$\sup_N \norm{\sum_{n=1}^N x_n}=\infty$ for every block sequence
$(x_n)$ of $(e_i)$ that is bounded away from $0$. The former is then obtained
by duality. Applying \cite[Corollary~2.4]{MR3742968} we then have that any
space with a basis admitting block upper $\ell_p$ estimates is AUS
(with power type $p$) and that any space admitting a basis with block
lower $\ell_q$ estimates is AUC$^*$ (with power type $q$) as the dual
of the space of $Y=[e_i^*]$.
As a consequence, a Banach space with a basis admitting block upper
$\ell_p$ estimates does not admit $\Delta$-points and the predual of a
Banach space with a basis admitting lower $\ell_q$ estimates does not
admit $\Delta$-points. This applies in particular to the predual of
the James tree space (see Section~\ref{sec:james-tree-space}) and to the Baernstein space $B$ (see at the end of Section~\ref{sec:asympt-near-unif}).

\section{Asymptotic uniform convexity}
\label{sec:asympt-near-unif}

The main result of this section is that
reflexive AUC spaces do not have $\Delta$-points.
The proof uses the characterization of
reflexive AUC spaces
in terms of the measure of non-compactness
of slices and relies on the following result of
Kuratowski (see e.g. \cite[p.~151]{MR3289625})

\begin{lem}\label{lem:Kuratowski}
  Let $(M,d)$ be complete metric space.
  If $(F_n)$ is a decreasing sequence of non-empty,
  closed, and bounded subsets of $M$ such that
  $\lim_{n} \alpha(F_n) = 0$, then the intersection
  $F_\infty = \bigcap_{n=1}^\infty F_n$ is a non-empty
  compact subset of $M$.
\end{lem}

Let us first proof a pointwise version of the main result.

\begin{thm}
  \label{thm:pointwise-weakalpha-nodelta}
  Let $X$ be a Banach space and $x \in S_X$.
  If there exists $x_0^* \in D(x)$ such that
  $\lim_{\delta \to 0} \alpha(S(x_0^*,\delta)) = 0$,
  then $x$ is not a $\Delta$-point.
\end{thm}

\begin{proof}
  Let $x \in S_X$ be such that
  there exists $x_0^* \in D(x)$ with
  $\lim_{\delta \to 0} \alpha(S(x_0^*,\delta)) = 0$ and let us 
  assume for contradiction that $x$ is a $\Delta$-point.

  Define a set of functionals norming $x$ by
  \begin{equation*}
    D_0(x) := \set{
    \frac{x^* + x_0^*}{2} : x^* \in D(x)
    }.
  \end{equation*}
  If $\delta > 0$ and $f \in D_0(x)$, then
  by Lemma~\ref{intersection_lemma}
  \begin{equation*}
    S\left(f,\frac{\delta}{2}\right)
    =
    S\left(\frac{x^* + x_0^*}{2},\frac{\delta}{2}\right)
    \subset
    S(x^*,\delta) \cap S(x_0^*,\delta)
  \end{equation*}
  and therefore $\lim_{\delta \to 0} \alpha(\overline{S}(f,\delta))
  = \lim_{\delta \to 0} \alpha(S(f,\delta)) = 0$
  for all $f \in D_0(x)$.

  Given $f \in D_0(x)$ and a sequence
  $(\delta_n) \subseteq (0,1)$ decreasing to $0$
  we define for each $n$
  \begin{equation*}
    F_n^{f} := \overline{S}(f,\delta_n) \cap \Delta_{\delta_n}(x).
  \end{equation*}
  Each $F_n^{f}$ is a closed non-empty subset of $B_X$
  with $\alpha(F_n^{f}) \to 0$ hence
  $F^f = \cap_n F_n^f$ is non-empty and compact
  for every $f \in D_0(x)$ by Lemma~\ref{lem:Kuratowski}. Using a compactness argument we can then prove the following.
  
  \bigbreak

  \textbf{Claim.}
  \begin{equation*}
    F := \bigcap_{f \in D_0(x)} F^f \neq \emptyset.
  \end{equation*}
  
  \bigbreak

  Before proving this claim let us see that it will
  give us the desired conclusion.
  Indeed, if $z \in F$, then $\|x - z\| = 2$ and
  for all $f \in D_0(x)$ we must have $f(z) = 1$,
  in particular $x^*_0(z) = 1$.
  But this is nonsense since for any $x^* \in S_{X^*}$
  with $x^*(x-z) = 2$ we must have $x^*(x) = 1$
  and $x^*(z) = -1$, so
  \begin{equation*}
    1 = f_{x^*}(x)=\frac{x^* + x^*_0}{2}(z)
    = \frac{-1 + 1}{2} = 0.
  \end{equation*}
  This contradiction shows that $x$ is not a $\Delta$-point.

  To finish the proof we only need to prove the
  claim that $F \neq \emptyset$. It is enough
  to show that $(F^{f})_{f \in D_0(x)}$
  has the finite intersection property
  since all the sets $F^f$ are compact and non-empty.

  Let $f_1,\ldots,f_k \in D_0(x)$,
  which means $f_j = \frac{x_j^* + x_0^*}{2}$ for
  $x^*_j \in D(x)$.
  Define
  \begin{equation*}
    f = \frac{1}{k} \sum_{j=1}^k f_j
    = \frac{x_0^* + \frac{1}{k}\sum_{j=1}^k x^*_j}{2}
  \end{equation*}
  Clearly $f \in D_0(x)$.
  By Lemma~\ref{intersection_lemma}
  we have for all $\delta > 0$
  \begin{equation*}
    S\left(f,\frac{\delta}{k}\right)
    \subset \bigcap_{j=1}^k S(f_j,\delta)
  \end{equation*}
  and hence
  \begin{equation}\label{eq:nuc-int}
    S\left(f,\frac{\delta}{k}\right)
    \cap \Delta_{\frac{\delta}{k}}(x)
    \subset
    \bigcap_{j=1}^k S(f_j,\delta)
    \cap
    \Delta_\delta(x).
  \end{equation}
  Since $(\delta_n)$ is decreasing there must for any $n$
  exist $m$ with $\delta_m < \delta_n/k$.
  By \eqref{eq:nuc-int} we get
  \begin{equation*}
    F_m^{f}
    \subset \bigcap_{j=1}^k F_n^{f_j}
  \end{equation*}
  and hence by commutativity of intersections
  \begin{equation*}
    \emptyset \neq \bigcap_n F_n^{f}
    \subseteq
    \bigcap_{j=1}^k
    \bigcap_{n} F_n^{f_j}
    =
    \bigcap_{j=1}^k F^{f_j}
  \end{equation*}
  and the claim is proved.
\end{proof}

From Theorem~\ref{thm:pointwise-weakalpha-nodelta}
we immediately get
\begin{thm}\label{thm:prop_alpha_no_Delta}
  If $X$ has Rolewicz' property $(\alpha)$, then $X$
  does not have $\Delta$-points.
\end{thm}

As we noted in Section~\ref{sec:preliminaries}
a Banach space $X$ is reflexive and AUC
if and only if it has uniform property $(\alpha)$.
Also finite-dimensional spaces are trivially AUC
since for example $\alpha(S(x^*,\delta)) = 0$
for slices of $B_X$ in finite-dimensional spaces.

\begin{thm}\label{thm:reflAUC_no_Delta}
  If $X$ is reflexive and AUC, then $X$ does not have $\Delta$-points.

  In particular, if $X$ is finite-dimensional
  then $X$ does not have $\Delta$-points.
\end{thm}

\begin{rem}
  Let $X$ be a Banach space such that for every
  $x \in S_X$ there exists $x^* \in D(x)$ with
  $\lim_{\delta \to 0} \alpha(S(x^*,\delta)) = 0$.
  Then by Theorem~\ref{thm:pointwise-weakalpha-nodelta}
  $X$ does not admit a $\Delta$-point.
  Note that unlike Rolewicz' property $(\alpha)$ (see \cite{MR879413})
  this property does not imply reflexivity.

  Indeed, every separable Banach space has an equivalent
  locally uniformly rotund renorming and
  if (the norm of) $X$ is locally uniformly rotund,
  then every $x \in S_X$ is strongly exposed by $x^* \in D(x)$
  so that for all $\varepsilon > 0$ there exists
  $\delta > 0$ such that $S(x^*,\delta)$ has
  diameter less than $\varepsilon$.
  In particular, $\alpha(S(x^*,\delta)) < \varepsilon$.
\end{rem}

Using the duality AUS/AUC in reflexive spaces we can in fact combine
Theorem~\ref{thm:AUS_no_delta_points} and \ref{thm:reflAUC_no_Delta}.

\begin{cor}
  If $X$ is reflexive and AUC, then neither $X$ nor $X^*$ admit $\Delta$-points.
\end{cor}

In particular we can apply this result to the Baernstein's space $B$ whose construction and basic properties are given in the introductory Chapter 0 of \cite[Construction 0.9]{CS}. This space was originally introduced in \cite{Bae} as an example of a reflexive space failing the Banach--Saks property. It is known to have a normalized (unconditional) basis with block lower $\ell_2$ estimates and thus to be $2$-AUC (see \cite[Theorem 3]{Partington} with the NUC terminology). Also the optimal modulus of near convexity of $B$ has been estimated in \cite{BOS}. From our preceding results the space $B$ and its dual space $B^*$ both fail to have $\Delta$-points.

Here is a pointwise application of Theorem~\ref{thm:pointwise-weakalpha-nodelta}.

\begin{cor}\label{cor:dual_auc_not_delta}
  Let $X$ be a Banach space and let $x^*\in S_{X^*}$ be a norm one
  functional which attains its norm at some $x\in S_X$.
  If $\lim_{t \to 0} \frac{\overline{\rho}(t,x)}{t} = 0$,
  then $x^*$ is not a $\Delta$-point.
\end{cor}

\begin{proof}
  By Corollary~\ref{cor:local_aus_kuc} we have
  $\lim_{\delta\to 0} \alpha(S(x,\delta)) = 0$
  and the conclusion follows directly from
  Theorem~\ref{thm:pointwise-weakalpha-nodelta} since $x\in D(x^*)$.
\end{proof}

Let $X$ be a Banach space.
An element $x \in B_X$ is said to be a \emph{quasi-denting point}
if given $\varepsilon > 0$, there exists $x^* \in S_{X^*}$ and
$\delta > 0$ with $x \in S(x^*,\delta)$ such that
$\alpha(S(x^*,\delta)) < \varepsilon$.
Recall that $x$ is a \emph{denting-point} if the above
can be strengthened to $S(x^*,\delta) \subset B(x,\varepsilon)$.
In dual spaces one can similarly define weak$^*$ (quasi-)denting
points by requiring $x^*$ to be weak$^*$ continuous.
Giles and Moors \cite{MR1188888} introduced quasi-denting points
(under the name $\alpha$-denting points)
see e.g. \cite{MR1277910} or \cite{MR1303494}.

Let $X$ be a Banach space such that the dual is AUC$^*$,
then every $x^* \in S_{X^*}$ is a quasi-denting point.
This is essentially contained in
e.g. \cite[Proposition~4.8]{MR2340479},
but we include the straightforward argument.
The AUC$^*$ is equivalent to the following version of UKK$^*$:
For every $\varepsilon > 0$ there exists $\delta > 0$
such that if $x^* \in B_{X^*}$ and all weak$^*$ neighborhoods
$V$ of $x^*$ satisfy $\diam(V \cap B_{X^*}) > \varepsilon$,
then $\|x^*\| \le 1 - \delta$.
Now fix $\varepsilon > 0$ and choose $\delta > 0$ as above.
If $x^* \in S_{X^*}$, then we can find $x \in S_X$ with
$x^*(x) > 1 - \delta$. Let $0 < \delta' < \delta$.
If $y^* \in S(x,\delta)$, then $\|y^*\| > 1 - \delta$
and the exits a weak$^*$ open neighborhood $V_{y^*}$ of $y^*$
with $\diam(V_{y^*} \cap B_{X^*}) \le \varepsilon$.
We therefore have an open cover of the weak$^*$ compact
set $\bar{S}(x,\delta')$ and by compactness we have
a finite cover and hence $x^*$ is a weak$^*$ quasi-denting point.

If the dual $X^*$ is AUC$^*$, then $X$ is AUS and by
Corollary~\ref{cor:dual_auc_not_delta} no norm-attaining
$x^* \in S_{X^*}$ can be a $\Delta$-point, but we do not know if a
weak$^*$-quasi-denting point, or more generally a quasi-denting point,
can be a Daugavet-point or a $\Delta$-point. For non-reflexive AUC
spaces, we do not even know if every element of the unit sphere is
quasi-denting.

\section{The James tree space}
\label{sec:james-tree-space}

Let $T=\{\emptyset\}\cup \bigcup_{n\geq 1} \{0,1\}^n$ be the infinite
binary tree and let us denote by $\leq$ the natural ordering on $T$.
A totally ordered subset $S$ of $T$ is called a segment if it
satisfies:
\begin{equation*}
  \forall s,t\in S,\ [s,t]=\{u\in T: s\leq u\leq t\} \subset S.
\end{equation*}
An infinite segment of $T$ is also called a branch of $T$. 
Let us denote by $\mathcal{F}$ the set of all
finite families of disjoint segments in $T$.
The James tree space $JT$ is define as follows:
\begin{equation*}
  JT = \left\{ x = (x_s)_{s\in T}
    \subset \Real:\ \norm{x}_{JT}^2= \sup_{F\in \mathcal{F}}
    \sum_{S\in F} x_S^2 <\infty \right\}
\end{equation*}
where $x_S = \sum_{s\in S}x_s$ for every non-empty segment $S$ and
$x_\emptyset = 0$.

It is well known that $JT$ is a Banach space and that the set of canonical
unit vectors $\{e_t\}_{t\in T}$ of $c_{00}(T)$ forms, for the
lexicographic order, a normalized monotone boundedly complete basis of
$JT$. Moreover, it is also known that the closed linear span
$JT_* = [e_t^*]_{t\in T}$ of the set of biorthogonal functionals in
$JT^*$ is a unique isometric predual of $JT$. If $S$ is a segment of
$T$ or if $\beta$ is a branch of $T$ we will sometimes refer to the set $\{e_s\}_{s\in S}$ as a segment of $JT$ and the set $\{e_t\}_{t\in \beta}$ as a branch of $JT$.
 From the definition of
the norm, it is easy to show that $\norm{x+y}_{JT}^2 \geq
\norm{x}_{JT}^2+\norm{y}_{JT}^2$ whenever $x,y\in JT$ have totally disconnected
supports (that is if $\conv(\supp_T x)\cap \conv(\supp_T
y)=\emptyset$). Since we are working with the lexicographic order on
$T$, this applies whenever $\supp x < \supp y$ with respect to the
ordering of the basis, and  $\{e_t\}_{t\in T}$ thus satisfies block
lower $\ell_2$ estimates.  It follows that $JT$ is $2$-AUC$^*$ and
Theorem~\ref{thm:AUS_no_delta_points} applied to the $2$-AUS $JT_*$
yields the following result.

\begin{thm}
  The predual $JT_*$ of the James tree space does not admit $\Delta$-points.
\end{thm}

Let us also emphasize that $\norm{x}_{JT}\geq \norm{x}_{\ell_2}$ for every $x\in JT$.  
It is also worth mentioning that if one consider the equivalent norm
$\norm{x}^2 = \sup_{S_1<\dots<S_n} \sum_{i=1}^n x_{S_i}^2$ on the James
space $J$, where the $S_i$ are segments of $\Natural$, then one
obtains a Banach space isometric to the closed linear span
$[e_t]_{t\in \beta}$ of any branch of $JT$. In particular all the
results we will obtain for the space $JT$ will also apply to
$(J,\norm{.})$.

As we see from Proposition~\ref{prop:unif_nsq_char}
the unit ball of spaces that are not uniformly non-square
contain small slices of diameter arbitrary close to $2$.
Let us start by illustrating this
by exhibiting slices of diameter $2$ in JT.

\begin{prop}
  For every $\delta>0$ we can find some $x^*\in S_{JT^*}$ such that
  $\diam S(x^*,\delta)=2$ and the diameter is attained.
\end{prop}

\begin{proof}
  For convenience let us work in the space $(J,\norm{.})$ introduced
  above. Doing the same construction on any branch $(x_t)_{t\in\beta}$
  of $JT$ would do the work for $JT$.

  Fix $\delta>0$, fix $n\geq 1$, and let
  \begin{equation*}
    x =
    \frac{1}{\sqrt{n}} \sum_{i=1}^n(-1)^{i+1}e_{2i-1}
    =
    \frac{1}{\sqrt{n}} (1,0,-1,0,\ldots, -1,0,0,0,\ldots)
  \end{equation*}
  and
  \begin{equation*}
    y
    =
    \frac{1}{\sqrt{n}}\sum_{i=1}^n(-1)^{i}e_{2i}
    =
    \frac{1}{\sqrt{n}} (0,-1,0,1,\ldots, 0,1,0,0,\ldots)
  \end{equation*}
  so that
  \begin{equation*}
    x - y = \frac{1}{\sqrt{n}} (1,1,-1,-1,\ldots,-1,-1,0,0,\ldots)
  \end{equation*}
  and
  \begin{equation*}
    x + y = \frac{1}{\sqrt{n}} (1,-1,-1,1,1,\ldots,-1,-1,1,0,0,\ldots).
  \end{equation*}

  It is easy to check that $\norm{x}=\norm{y}=1$, $\norm{x-y}=2$, and
  $\norm{x+y}=\sqrt{\frac{1+(n-1)\times2^2+1}{n}}=\sqrt{4-\frac{2}{n}}$.

  By assuming that $n$ was chosen large enough so that
  $\sqrt{4 - \frac{2}{n}} > 2 - \delta$ and by choosing a norming
  functional $x^*\in S_{J^*}$ for $x + y$ we then have
  $x^*(x + y) = \norm{x + y} > 2 - \delta$ and this implies that
  $x, y \in S(x^*,\delta)$.
  Since $\norm{x - y} = 2$ the conclusion follows.
\end{proof}

We start by showing that the elements of the basis of $JT$
are weak$^*$ denting (and even strongly exposed by an element of the
predual) and then use this to prove that $JT$ does not admit weak$^*$ Daugavet-points.

\begin{lem}
  For each $t\in T$, $\lim_{\delta\to 0} \diam S(e_t^*, \delta) = 0$.
\end{lem}

\begin{proof}
  Let us fix $t\in T$ and $\delta>0$, and let us take
  $y \in S(e_t^*,\delta)$.
  By the triangle inequality, we have
  \begin{equation*}
    \norm{y - e_t}_{JT} \leq \norm{y - y_te_t}_{JT} + (1 - y_t)
    \leq \norm{y-y_te_t}_{JT} + \delta
  \end{equation*}
  so we only need to estimate $\norm{y - y_te_t}_{JT}$.

  If $t$ is not in the support of a family $F\in \mathcal{F}$, that is
  if $t\notin \bigcup_{S\in F} S$, then the segment $\{t\}$ is
  disjoint from all the segments in $F$ and we
  have
  \begin{equation*}
    y_t^2 + \sum_{S\in F} (y - y_te_t)_{S}^2
    = y_t^2 + \sum_{S\in F} y_{S}^2\leq\norm{y}_{JT}^2
    \leq 1
  \end{equation*}
  so that
  \begin{equation*}
    \sum_{S\in F} (y-y_te_t)_{S}^2 \leq 1 - (1 - \delta)^2
    \leq 2 \delta.
  \end{equation*}

  Now if we take any segment $S$ of $T$ containing $t$, we can write
  $S=S^- \cup \{t\} \cup S^+$ with $S^- < \{t\} < S^+$ segments of
  $T$, and by the preceding computations we have
  \begin{equation*}
    (y-y_te_t)_S^2
    \leq y_{S^-}^2 + y_{S^+}^2 + 2\sqrt{y_{S^-}^2y_{S^+}^2}
    \leq 6\delta.
  \end{equation*}
  By combining the two observations we get
  $\sum_{S\in F} (y - y_te_t)_{S}^2 \leq 8 \delta$
  for every $F \in \mathcal{F}$,
  that is $\norm{y - y_te_t}_{JT}^2 \leq 8\delta$.
  The conclusion follows.
\end{proof}

\begin{cor}
  The space $JT$ does not admit weak$^*$ Daugavet-points.
\end{cor}

\begin{proof}
  Let us fix $x\in S_{JT}$. If $x=e_t$ for some $t\in T$ then $x$ is a
  weak$^*$ denting point by the preceding lemma and $x$ cannot be a
  weak$^*$ Daugavet-point. So let us assume that $x$ is not of that
  form and let us take $t\in T$ such that $x_t\neq 0$ and $x_s=0$ for
  every $s<t$.

  Now let us write $x_t=\theta\alpha$ with $\alpha\in (0,1)$ and
  $\theta\in \{-1,1\}$. Because of the choice of $t$ we clearly
  have
  \begin{equation*}
    \norm{x-x_te_t}_{JT}^2
    =
    \sup_{F\in \mathcal{F}_t } \sum_{S\in F}^n x_S^2
  \end{equation*}
  where $\mathcal{F}_t$ is the set of finite families of disjoint
  segments of $T$ not intersecting $[\emptyset, t]$.
  Since
  \begin{equation*}
    x_t^2 + \sum_{S\in F} x_S^2
    \leq \norm{x}_{JT}^2
  \end{equation*}
  for every such family we obtain
  \begin{equation*}
    \norm{x - x_te_t}_{JT}^2
    \leq 1 - \alpha^2 < 1.
  \end{equation*}
  Thus
  \begin{equation*}
    \norm{x - \theta e_t}_{JT}
    \leq \norm{x - x_te_t}_{JT} + (1 - \alpha)
    \leq 2 - \alpha
  \end{equation*}
  so that $x$ is at distance strictly less than $2$ from a (weak$^*$)
  denting point of $JT$.  The conclusion follows from a weak$^*$
  version of \cite[Proposition~3.1]{JRZ}.
\end{proof}

Now we will show that $JT$ does not admit $\Delta$-points. For this
purpose let us introduce some more notations. Let us write
$\mathcal{F}_\infty$ for the set of (finite or infinite) families of
disjoint segments in $T$. Note that every infinite
$F\in\mathcal{F}_\infty$ has to be countable.
For any segment $S$ of $T$, let us write
$\mathbbm{1}_S = \sum_{s \in S}e_s^*$.
A molecule in $JT^*$ is an element of the form
$\sum_{i \geq 1}\lambda_i \mathbbm{1}_{S_i}$
where $\{S_i\}\in \mathcal{F}_{\infty}$
and $\lambda = (\lambda_i) \in B_{\ell_2}$.
The molecule is finite if the family $\{S_i\}$
contains only finitely many non-empty segments, and we will write
$\mathcal{M}$ (resp. $\mathcal{M}_\infty$) for the set of finite
(resp. finite or infinite) molecules of $JT^*$.

Molecules play an important role in the study of the space $JT$
because they turn computations in $JT$ into computations in $\ell_2$
and because of the following result due to Schachermayer
\cite[Proposition~2.2]{SchJT}.

\begin{thm}
  \label{thm:BJT*=clcoM}
  The unit ball $B_{JT^*}$ of the dual of JT is the norm closed convex
  hull of $\mathcal{M}$.
\end{thm}

\begin{rem}
  \label{rem:weak*-clM=Minfty}
  It is also known that $\mathcal{M}_\infty$ is the weak$^*$ closure
  of $\mathcal{M}$ in $B_{JT^*}$.
\end{rem}

We will use some specific molecules to provide norming functionals for
elements of $S_{JT}$. For $x\in S_{JT}$ and $F\in\mathcal{F}_\infty$,
let us write $m_{x,F} = \sum_{S\in F}x_S\mathbbm{1}_S$.  Since
$\sum_{S\in F}x_S^2 \leq \norm{x}_{JT}=1$, those elements are
molecules in $JT^*$. Moreover, we have the following result.

\begin{lem}\label{l_2_lemma}
  Let us assume that $\sum_{S\in F} x_S^2 = \norm{x}_{JT} = 1$.
  Then the molecule $m_{x,F}$ belongs to $D(x)$ (that is, it is a norm
  one functional and it norms $x$). Moreover, if $y\in S(m_{x,F},\delta)$
  for some $\delta>0$, then we have
  \begin{equation*}
    \sum_{S\in F} (x_S - y_S)^2 \leq 2\delta.
  \end{equation*}
\end{lem}

\begin{proof}
  The key observation here is that for every $y\in JT$,
  we have
  \begin{equation*}
    m_{x,F}(y) = \sum_{S\in F} x_Sy_S
    = \langle (x_S)_{S\in F}, (y_S)_{S\in F} \rangle_{\ell_2}.
  \end{equation*}
  In particular we have, by the Cauchy--Schwartz inequality,
  \begin{equation*}
    \abs{m_{x,F}(y)} \leq
    \norm{(x_S)_{S\in F}}_{\ell_2} \norm{ (y_S)_{S\in F}}_{\ell_2}
    \leq \norm{x}_{JT} \norm{y}_{JT},
  \end{equation*} so
  $\norm{m_{x,F}}_{JT^*} \leq 1$.
  By assumption
  $m_{x,F}(x) = \norm{(x_S)_{S\in F}}_{\ell_2}^2 = 1$
  and the first part of the lemma follows.

  Now $y\in S(m_{x,F},\delta)$ if and only if
  $\langle (x_S)_{S\in F}, (y_S)_{S\in F} \rangle_{\ell_2}
  > 1-\delta$ and thus every
  $y$ in $S(m_{x,F},\delta)$ satisfy
  \begin{align*}
    \norm{(x_S-y_S)_{S\in F}}_{\ell_2}^2
    =&
    \norm{ (x_S)_{S \in F}}_{\ell_2}^2
    +
    \norm{ (x_S)_{S \in F}}_{\ell_2}^2
    -
    2 \langle (x_S)_{S\in F}, (y_S)_{S\in F} \rangle_{\ell_2} \\
    & \leq 2 - 2(1 - \delta)
    = 2 \delta. \qedhere
  \end{align*}
\end{proof}

Note that it is not obvious at first that such a norm attaining family
exists. We will first do a warm up with finitely supported elements in
$JT$ for which this is obvious, and then prove it in a lemma.

\begin{prop}
  Let $x\in S_{JT}$ be an element of finite support. Then $x$ is not a
  weak$^*$ $\Delta$-point.
\end{prop}

\begin{proof}
  Let $x\in S_{JT}$ be an element of finite support $\Sigma$ and let
  $\mathcal{F}_\Sigma$ be the set of families of disjoint segments of
  the convex hull of $\Sigma$ in  $T$ (that is the smaller subset of
  $T$ containing $[s,t]$ for every $s\leq t$ in $\Sigma$). Then
  $\mathcal{F}_\Sigma$ is finite and we have
  \begin{equation*}
    \norm{x}_{JT}^2
    =
    \max_{F \in \mathcal{F}_\Sigma} \sum_{S\in F}x_S^2.
  \end{equation*}
  Now let
  $\mathcal{D} =
  \{ F\in \mathcal{F}_\Sigma :\ \sum_{S\in F}x_S^2=1\}$.
  The set $\mathcal{D}$ is a non-empty subset of $\mathcal{F}_\Sigma$
  and since this set is finite we can find a constant $\eta_x>0$ such
  that for every $F \in \mathcal{F}_\Sigma \backslash \mathcal{D}$
  we have
  \begin{equation*}
    \sum_{S\in F}x_S^2 < (1 - \eta_x)^2.
  \end{equation*}

  Let us introduce
  $x^* = \frac{1}{\abs{\mathcal{D}}} \sum_{F\in \mathcal{D}} m_{x,F}$
  where $m_{x,F}$ is the molecule associated to $x$ and $F$.
  For every $F\in \mathcal{D}$, the molecule $m_{x,F}$ is in $D(x)$
  by the preceding lemma, and by Lemma~\ref{intersection_lemma}
  we know that $x^*$ is in $D(x)$ and that for every choice of
  a $\delta>0$ we have
  \begin{equation*}
    S\left(x^*, \frac{\delta}{n}\right)
    \subset \bigcap_{F\in \mathcal{D}} S(m_{x,F},\delta).
  \end{equation*}
  To conclude choose any $\delta\in(0,1)$ and let us take
  $y \in S\left(x^*, \frac{\delta}{n}\right)$.
  If $F \in \mathcal{F}$ does not belong to $\mathcal{D}$,
  then by Minkowski's inequality we have
  \begin{equation*}
    \sum_{S\in F} (x_S - y_S)^2
    \leq
    \left(\sqrt{\sum_{S\in F}x_S^2}
      +
      \sqrt{\sum_{S\in F}y_S^2}
    \right)^2
    \leq (2 - \eta_x)^2.
  \end{equation*}
  Now if $F\in \mathcal{D}$ we have $y \in S(m_{x,F},\delta)$
  and by the preceding lemma we get
  \begin{equation*}
    \sum_{S\in F} (x_S - y_S)^2 \leq 2\delta.
  \end{equation*}
  Finally
  $\norm{x - y}_{JT} \leq \max\{ 2 - \eta_x, \sqrt{2 \delta}\}$
  and this quantity is strictly less than $2$ since $\delta<1$.
  To conclude, note that $x^*\in JT_*$ since all the segments
  involved in the construction are finite.
\end{proof}

To tackle the other elements of $JT$ we first need to ensure the
existence of norm attaining (possibly infinite) families.

\begin{lem}
  Let $x\in S_{JT}$. 
  Then there is an $F \in \mathcal{F}_\infty$
  such that $\sum_{S\in F} x_S^2 = 1$.
\end{lem}

\begin{proof}
  Let $x \in S_{JT}.$ By Krein--Milman $x$ attains its norm on an
  extreme point $x^* \in B_{JT*}.$
  By Milman's converse, Lemma~\ref{thm:BJT*=clcoM}
  and Remark~\ref{rem:weak*-clM=Minfty}, $x^* \in M_\infty,$
  so we can write $x^*= \sum_{i\geq 1} \lambda_i\mathbbm{1}_{S_i}$
  for some $\lambda = (\lambda_i)$ in $B_{\ell_2}$
  and $\{S_i\}$ in $\mathcal{F}_{\infty}$.
  To conclude, consider $\mu=(x_{S_i})$ in $B_{\ell_2}$
  and observe that
  \begin{equation*}
    \langle \lambda,\mu\rangle_{\ell_2}
    =x^*(x)
    =1,
  \end{equation*}
  so $\lambda$ strongly exposes $\mu$ in $B_{\ell_2}.$ Thus
  $\lambda=\mu$ and $\norm{\mu}_{\ell_2} =1$  which is precisely the
  desired result.
\end{proof}

For every $x \in S_{JT}$ let us introduce
$\mathcal{D}(x) =
\{ F \in \mathcal{F}_\infty :\ \sum_{S \in F}x^2_S = 1 \}$.
By the preceding lemma $\mathcal{D}(x)$ is non-empty, but it
does not need to be finite. To get around this problem we will
consider the restriction of families in $\mathcal{D}(x)$ to
a subtree of finite level.
So for every $N\geq 1$ let us write
$T_N = \text{level}[0,N]
= \{\emptyset\} \cup \bigcup_{n=1}^N \{0,1\}^n$
the binary tree of height $N$ and let us write
$\mathcal{D}_N(x) = \{F\cap T_N :\ F\in \mathcal{D}(x)\}$
where $F\cap T_N = \{S\cap T_N :\ S\in F\}$.
Then $\mathcal{D}_N(x)$ is a finite non-empty
set and we have, similarly to the finite support
case, the following lemma.

\begin{lem}\label{D_N_lemma}
  For every $x \in S_{JT}$ and for every $N\geq 1$, there is a constant
  $\eta_{x,N} > 0$ such that $\sum_{S\in F} x_S^2 < (1-\eta_{x,N})^2$
  for every $F$ in $\mathcal{F}$ for which $F\cap T_N$ does not belong
  to $\mathcal{D}_N(x)$.
\end{lem}

\begin{proof}
  If this was not true, then since $\mathcal{F}\cap T_N$ is also
  finite we could find a family $F \in \mathcal{F}$ of segments of
  $T_N$ not belonging to $\mathcal{D}_N(x)$ and a sequence
  $(F_i)_{i\geq 1} \subset \mathcal{F}$ such that
  $F_i\cap T_N = F$ and
  $\sum_{S \in F_i}x_S^2 > \left( 1 - \frac{1}{i} \right)^2$
  for every $i\geq 1$. Using a compactness argument
  (weak$^*$ extractions for the sequence of corresponding
  molecules in $B_{JT^*}$ and metrizability of the weak$^*$ topology on
  $B_{JT^*}$) 
  we would then obtain a family $G\in \mathcal{F}_\infty$ for which
  $\sum_{S\in G}x_S^2=1$  and such that $G\cap T_N= F$. But then this
  would mean that $F\in \mathcal{D}_N(x)$ and we would get a
  contradiction.
\end{proof}

We will need a last easy fact which states that any $x\in S_{JT}$ has
its norm almost concentrated on a subtree of finite level. The proof
is elementary and comes from the definition of the norm of $JT$.

\begin{lem}
  Let $x \in S_{JT}$ and let $\varepsilon>0$.
  There is an $N\geq 1$ such that for every family
  $F \in \mathcal{F}_\infty$ of segments of $T$ which
  do not intersect $T_N$ one has
  $\sum_{S\in F}x_S^2 \leq \varepsilon^2$.
\end{lem}

With those tools in hand we can now prove the main result of this section.

\begin{thm}
  Let $x\in S_{JT}$. Then $x$ is not a $\Delta$-point.
\end{thm}

\begin{proof}
  Let us fix some $x\in S_{JT}$ and let us assume that $x$ has
  infinite support. Let us also fix some $\varepsilon > 0$ and let us
  take some $N\geq 1$ for which $x$ is almost concentrated on $T_N$ in
  the sense of the previous lemma.

  For every $F$ in  $\mathcal{D}_N(x)$ we pick a representative family
  $F_\mathcal{R}$ in $\mathcal{D}(x)$ which satisfy
  $F_\mathcal{R} \cap T_N = F$. This means that for every
  $S \in F$ there is (a unique) $S_\mathcal{R} \in F_\mathcal{R}$
  such that $S_\mathcal{R}\cap T_N = S$. For every such $F$
  we define $x_F^*=m_{x,F_\mathcal{R}}$ to be the molecule associated
  to $x$ and $F_\mathcal{R}$ and we let $x^*$ be the average of
  the $x_F^*s$ with $F \in \mathcal{D}_N(x)$.
  Since each
  $x_F^*$ is in $D(x)$, Lemma~\ref{intersection_lemma} tells us that
  $x^*\in D(x)$ and that
  \begin{equation*}
    S\left(x^*, \frac{\varepsilon^2}{2n}\right)
    \subset \bigcap_{F\in \mathcal{D}_N(x)}
    S\left(x_F^*,\frac{\varepsilon^2}{2}\right)
  \end{equation*}
  for $n = \abs{\mathcal{D}_N(x)}$.
  We now want to get a uniform bound for
  $\norm{x - y}_{JT}$ on
  $S\left(x^*, \frac{\varepsilon^2}{2n}\right)$.

  So let us take
  $y \in S\left(x^*, \frac{\varepsilon^2}{2n}\right)$
  and let us fix some $G\in \mathcal{F}$.
  First observe that Lemma~\ref{D_N_lemma} allows us to get rid of the
  case $G \cap T_N\notin\mathcal{D}_N(x)$ exactly as in the finite
  support proof because it yields
  \begin{equation*}
    \sum_{S\in G}(x_S-y_S)^2
    \leq
    (2 - \eta_{x,N})^2.
  \end{equation*}
  So let us assume that $F = G\cap T_N$ belongs to $\mathcal{D}_N(x)$.
  We will split the family $G$ into $4$ disjoint subfamilies
  $G_i$, $1\leq i\leq 4$, and we will estimate separately the sums
  $\sum_{S\in G_i}(x_S-y_S)^2$. For this
  we will use repeatedly the two following inequalities.

\bigbreak

\textbf{Claim $A$.}
Let $H$ be a subfamily of $F_\mathcal{R}$.
Then
\begin{equation*}
  \sum_{S\in H} (x_S - y_S)^2 \leq  \varepsilon^2.
\end{equation*}

\bigbreak

\begin{proof}[Proof of claim $A$]
  Since segments in $H$ are in $F_\mathcal{R}$
  and since $y \in S\left(x_F^*, \frac{\varepsilon^2}{2}\right)$
  with $x_F^* = m_{x,F_\mathcal{R}}$
  Lemma~\ref{l_2_lemma} yields
  \begin{equation*}
    \sum_{S\in H} (x_S - y_S)^2
    \leq \sum_{S\in F_\mathcal{R} }(x_S - y_S)^2
    \leq \frac{2\varepsilon^2}{2}=\varepsilon^2.
  \end{equation*}
\end{proof}

\bigbreak

\textbf{Claim $B$.}
Let $H$ be a family of disjoint segments of $T$ which do not intersect
$T_N$. Then
\begin{equation*}
  \sum_{S\in H} (x_S - y_S)^2 \leq
  \sum_{S\in H}y_S^2  + \varepsilon^2 + 2\varepsilon.
\end{equation*}

\bigbreak

\begin{proof}[Proof of claim $B$]
  Since segments in $H$ do not intersect $T_N$ we have,
  by our initial choice of $N$,
  $\sum_{S\in H} x_S^2 \leq \varepsilon^2$ and so using Minkowski's inequality
  \begin{align*}
    \sum_{S\in H} (x_S-y_S)^2
    &\leq
    \sum_{S\in H} x_S^2 + \sum_{S\in H}y_S^2
    + 2\sqrt{\sum_{S\in H} x_S^2} \sqrt{\sum_{S\in H}y_S^2} \\
    &\leq
    \sum_{S\in H}y_S^2 + \varepsilon^2 + 2\varepsilon.
\end{align*}
\end{proof}

Now let us split the family $G$ and let us start the computations of
the corresponding sums.

\bigbreak

\textbf{Claim 1.}
Let $G_1 =\{S \in G:\ S\subset T_{N-1}\}$.
There is a $\gamma_1 = \gamma_1(\varepsilon)$ such that
\begin{equation*}
  \sum_{S \in G_1} (x_S - y_S)^2
  \leq \gamma_1.
\end{equation*}

\bigbreak

\begin{proof}[Proof of Claim~1]
  Since segments in $G_1$ are contained in $T_{N-1}$, every segment $S$
  in $G_1$ has to be equal to its representative $S_\mathcal{R}$
  because $S_\mathcal{R} \cap T_N = S$. This means that $G_1$ is a
  subfamily of $F_\mathcal{R}$ and the result follows directly from
  Claim~$A$.
\end{proof}

\bigbreak

\textbf{Claim 2.}
Let $G_2 = \{S\in G:\ S\subset T\backslash T_N\}$.
There is a $\gamma_2 = \gamma_2(\varepsilon)$ such that
\begin{equation*}
  \sum_{S\in G_2} (x_S - y_S)^2
  \leq \sum_{S\in G_2} y_S^2  + \gamma_2.
\end{equation*}

\bigbreak

\begin{proof}[Proof of Claim~2]
  Since segments in $G_2$ do not intersect $T_N$ the result follows
  directly from Claim~$B$.
\end{proof}

The remaining segments of $G$ are those that intersect the
$N^{\text{th}}$ level of $T$. For those, we will distinguish between
segments $S$ such that $S$ and its representative $S_\mathcal{R}$ are
on the same branch of $T$ and those for which $S$ and $S_\mathcal{R}$ split at some higher level.  So let $G_3$ be the subset of the remaining segments of $G$ satisfying the first condition and let $G_4$ be the subset of those satisfying the second condition.

Note that if $S$ is a segment in $G_3$, then since $S$ and
$S_\mathcal{R}$ are on the same branch and are equal in $T_N$ we have
either $S\subset S_{\mathcal{R}}$ or $S_\mathcal{R}\subset S$ and the
complements are in either case segments which do not intersect
$T_N$. Thus.

\bigbreak

\textbf{Claim 3.}
Let $G_{3,1} = \{ S\in G_3:\ S = S_{\mathcal{R}}\}$,
let $G_{3,2} = \{ S\in G_3:\ S \subsetneqq S_{\mathcal{R}}\}$
and let $G_{3,3}=\{ S\in G_3:\ S_\mathcal{R} \subsetneqq S\}$.
There is a $\gamma_3 = \gamma_3(\varepsilon)$ such that
\begin{equation*}
  \sum_{S\in G_3} (x_S - y_S)^2
  \leq
  \sum_{S\in G_{3,2}} y_{S_\mathcal{R}\backslash S}^2
  + \sum_{S\in G_{3,3}} y_{S\backslash S_\mathcal{R}}^2
  + \gamma_3.
\end{equation*}

\bigbreak

\begin{proof}[Proof of Claim~3]
  By definition $G_{3,1}$ is a subfamily of $F_{\mathcal{R}}$
  so Claim~$A$ yields
  \begin{equation*}
    \sum_{S\in G_{3,1}} (x_S - y_S)^2 \leq \varepsilon^2.
  \end{equation*}
  Now let us deal with $G_{3,2}$.
  By Claim~$A$ we know that
  \begin{equation*}
    \sum_{S\in G_{3,2}} (x_{S_\mathcal{R}} - y_{S_\mathcal{R}})^2
    \leq \varepsilon^2.
  \end{equation*}
  Moreover, for every $S\in G_{3,2}$, we know that the segment
  $S_\mathcal{R} \backslash S$ does not intersect $T_N$ and
  thus Claim~$B$ yields
  \begin{equation*}
    \sum_{S\in G_{3,2}}
    (x_{S_\mathcal{R}\backslash S} - y_{S_\mathcal{R}\backslash S})^2
    \leq
    \sum_{S\in G_{3,2}} y_{S_\mathcal{R}\backslash S}^2
    + \varepsilon^2 + 2\varepsilon.
  \end{equation*}
  So finally we get
  \begin{align*}
    \sum_{S\in G_{3,2}} (x_S - y_S)^2
    &=
    \sum_{S\in G_{3,2}}
    [
    (x_{S_\mathcal{R}} - y_{S_\mathcal{R}})
    - (x_{S_\mathcal{R}\backslash S} - y_{S_\mathcal{R}\backslash S})
    ]^2 \\
    &\leq
    \sum_{S\in G_{3,2}}
    (x_{S_\mathcal{R}} - y_{S_\mathcal{R}})^2
    +
    \sum_{S \in G_{3,2}}
    (x_{S_\mathcal{R}\backslash S} - y_{S_\mathcal{R}\backslash S})^2 \\
    &\quad\quad\quad
    +
    2\sqrt{
      \sum_{S\in G_{3,2}}(x_{S_\mathcal{R}} - y_{S_\mathcal{R}})^2
    }
    \sqrt{
      \sum_{S\in G_{3,2}}
      ( x_{S_\mathcal{R}\backslash S} -
      y_{S_\mathcal{R}\backslash S} )^2
    }
    \\
    &\leq
    \sum_{S \in G_{3,2}} y_{S_\mathcal{R}\backslash S}^2
    + \varepsilon^2 + (\varepsilon^2 + 2\varepsilon)
    + 4\varepsilon.
  \end{align*}
  Similar computations allow us to deal with $G_{3,3}$ and the
  conclusion follows by combining the $3$ inequalities.
\end{proof}

Finally let us take $S\in G_4$. Since $S$ and its representative
$S_{\mathcal{R}}$ are equal on $T_N$ the segments
$S^- = S \backslash (S\cap S_\mathcal{R})$ and
$S_\mathcal{R}^- =
S_\mathcal{R} \backslash (S \cap S_\mathcal{R})$
do not intersect $T_N$ (and are non-empty). Thus.
\bigbreak

\textbf{Claim 4.}
There is a $\gamma_4 = \gamma_4(\varepsilon)$ such that
\begin{equation*}
  \sum_{S\in G_4} (x_S - y_S)^2
  \leq
  1 + \sum_{S\in G_4} y_{S^-}^2
  + \sum_{S\in G_4} y_{S_\mathcal{R}^-}^2 + \gamma_4.
\end{equation*}

\bigbreak

\begin{proof}[Proof of Claim~4]
  Using again Claim~$A$ and Claim~$B$, we have
  \begin{equation*}
    a=\sum_{S\in G_4} (x_{S_\mathcal{R}} - y_{S_\mathcal{R}})^2
    \leq \varepsilon^2,
  \end{equation*}
  \begin{equation*}
    b=\sum_{S\in G_4} (x_{S^-} - y_{S^-})^2
    \leq
    \sum_{S\in G_4} y_{S^-}^2 + \varepsilon^2 + 2\varepsilon,
  \end{equation*}
  and
  \begin{equation*}
    c=\sum_{S\in G_4} (x_{S_\mathcal{R}^-} - y_{S_\mathcal{R}^-})^2
    \leq \sum_{S\in G_4} y_{S_\mathcal{R}^-}^2
    + \varepsilon^2 + 2\varepsilon.
  \end{equation*}
  Thus we get, using the $\sqrt{u+v}\leq \sqrt{u}+\sqrt{v}$
  identity for the last term:
  \begin{align*}
    \sum_{S\in G_4} (x_S - y_S)^2
    &=
    \sum_{S\in G_4}
    [
    (x_{S_\mathcal{R}} - y_{S_\mathcal{R}})
    - (x_{S_\mathcal{R}^-} - y_{S_\mathcal{R}^-})
    + (x_{S^-} - y_{S^-})
    ]^2 \\
    &\leq  a+b+c+2\sqrt{ab}+2\sqrt{ac}+2\sqrt{bc} \\
    &\leq \eps^2+ \sum_{S\in G_4} y_{S^-}^2
    + \sum_{S\in G_4} y_{S_\mathcal{R}^-}^2+ 2(\varepsilon^2 + 2\varepsilon)
    + 4 \sqrt{\eps^2(1 + \varepsilon^2 + 2\varepsilon)} \\ 
    &\quad\quad\quad+ 2\sqrt{\sum_{S\in G_4} y_{S^-}^2}
    \sqrt{ \sum_{S\in G_4} y_{S_\mathcal{R}^-}^2}+ 2\sqrt{2(\varepsilon^2 + 2\varepsilon)
      +(\varepsilon^2 + 2\varepsilon)^2}
  \end{align*}
  so it only remains to show that
  \begin{equation*}
    2\sqrt{\sum_{S\in G_4} y_{S^-}^2}
    \sqrt{ \sum_{S\in G_4} y_{S_\mathcal{R}^-}^2} \leq 1.
  \end{equation*}
  To show this, observe that if you take two different segments $S$
  and $T$ in $G_4$, then you obviously have $S \cap T = \emptyset$ and
  $S_\mathcal{R} \cap T_\mathcal{R} = \emptyset$ since we work with
  families of disjoint segments, but you also have
  $S \cap T_\mathcal{R} = \emptyset$ and
  $T \cap S_\mathcal{R} = \emptyset$ since
  $S$ and $T$ have disjoint starting parts in $T_N$.
  Consequently:
  \begin{equation*}
    \sum_{S\in G_4} y_{S^-}^2
    +
    \sum_{S\in G_4} y_{S_\mathcal{R}^-}^2
    \leq
    \norm{y}_{JT}^2=1.
  \end{equation*}
  The conclusion follows since the function $f(x)=x\sqrt{1-x^2}$ has
  maximum $\frac{1}{2}$ on $[0,1]$ (attained at $\frac{1}{\sqrt{2}}$).
\end{proof}

Now letting $\gamma = \sum_{i=1}^4\gamma_i$
and combining the results from the $4$ claims, we obtain:
\begin{align*}
  \sum_{S\in G}(x_S - y_S)^2
  &\leq
  1 + \sum_{S\in G_2} y_S^2
  + \sum_{S\in G_{3,3}} y_{S\backslash S_\mathcal{R}}^2
  + \sum_{S\in G_4} y_{S^-}^2 \\
  &\quad\quad\quad
  + \sum_{S\in G_{3,2}} y_{S_\mathcal{R}\backslash S}^2
  + \sum_{S\in G_4} y_{S_\mathcal{R}^-}^2
  +\gamma,
\end{align*}
and since we work with families of disjoint segments we get
\begin{equation*}
  \sum_{S\in G}(x_S - y_S)^2 \leq 3 + \gamma.
\end{equation*}
Finally,
$\norm{x - y}_{JT} \leq
\max\{ 2-\eta_{x,N} \sqrt{3 + \gamma}\}$
and this quantity is strictly less than $2$ if we take a small enough
$\varepsilon$ since $\gamma$ goes to $0$ as $\varepsilon$ goes to $0$.
\end{proof}

\begin{rem}
  Here the segments involved can be infinite so the functional $x^*$ we
  slice with does not have to be in $JT_*$. This raises the following
  question.
\end{rem}

\begin{quest}
  Does $JT$ admit weak$^*$ $\Delta$-points?
\end{quest}

From the work of \cite{Gir} we know that the dual of the James tree
space $JT^*$ is AUC. Now $JT$ does not admit an equivalent norm whose dual norm is AUC$^*$. Indeed such norm would be AUS by the asymptotic duality. This is impossible because $JT$ is not Asplund (it
is separable while $JT^*$ is not), see
\cite[Proposition~2.4]{MR1888429}. In view of those observations
$JT^*$ would be a natural candidate for our study and the question of
the existence of $\Delta$- or Daugavet-points there, as well as the
question of the existence of non quasi-denting points would be
particularly relevant. Unfortunately it seems that computations are
still out of reach in this context, although the use of molecules
facilitates them, and we where only able to obtain a few partial
results even while trying to restrict ourselves to $J^*$.

\begin{lem}\label{predual-norm-est}
  Let $s,t\in T$ be two distinct points. For any $\varepsilon > 0$ and
  $z^*\in JT^*$ with $\{s,t\} \cap \supp(z^*) = \emptyset$ we have
  \begin{equation*}
    \|e_t^* - \varepsilon e_s^* + z^*\| > 1.
  \end{equation*}
\end{lem}

\begin{proof}
  Let $x^* = e_t^* - \varepsilon e_s^* + z^*$ for some $\varepsilon >
  0$ and $z^*$ supported in $T\backslash\{s,t\}$. Assume first that
  $\varepsilon \ge 1$, and
  let $\alpha = 1/\sqrt{2}$ and $x = \alpha e_t - \alpha e_s$.
  We have $\|x\| = 1$ and
  \begin{equation*}
    x^*(x) = \alpha + \varepsilon \alpha
    \geq 2/\sqrt{2} > 1.
  \end{equation*}
  Next if $\varepsilon < 1$ we let
  $\alpha = \sqrt{1-\varepsilon^2}$
  and $x = \alpha e_t - \varepsilon e_s$.
  We have $\|x\| = 1$ since
  $\alpha^2 + \varepsilon^2 = 1$
  and $(\alpha - \varepsilon)^2
  \le 1 - 2\alpha\varepsilon < 1$.
  Now
  \begin{equation*}
    x^*(x) = \alpha + \varepsilon^2
    > \alpha^2 + \varepsilon^2 = 1.
  \end{equation*}
  In both cases, $\|x^*\| > 1$.
\end{proof}

\begin{lem}
  In $JT^*$ every basis vector
  $e_t^*$ is an extreme point of $B_{JT^*}$,
  and therefore a weak$^*$ denting point.
\end{lem}

\begin{proof}
  Assume that $x^*$ and $y^*$ are norm one
  elements such that $e_t^* = \frac{x^*+y^*}{2}$.
  Then we have $x^*(e_t) = y^*(e_t) = 1$
  and $x^*(e_s) = - y^*(e_s)$ for all $s \neq t$.
  If we have
  $x^*(e_s) < 0$ for some $s \neq t$,
  then with $z^* = x^* - (e^*_t + x^*(e_s)e^*_s)$
  we get
  \begin{equation*}
    \|x^*\| = \|e^*_t + x^*(e_s)e^*_s + z^*\| > 1
  \end{equation*}
  by Lemma~\ref{predual-norm-est}.
  Hence $x^* = y^* = e^*_t$.

  By Proposition~3.d.19 in \cite{MR1474498}
  we know that $e^*_t$ is a point of weak$^*$
  to norm continuity on the unit ball of $JT^*$ so the conclusion
  follows by applying Choquet's lemma (see for example
  \cite[Lemma~3.69]{MR2766381}) which tells that the weak$^*$ slices
  form a neighborhood basis of $e_t^*$. Indeed the continuity of the
  identity map at $e_t^*$ then ensures that any ball around $e_t^*$
  contains a weak$^*$ slice containing $e_t^*$.
\end{proof}

\begin{cor}
  No molecule in $JT^*$ is a weak$^*$ Daugavet point.
\end{cor}

\begin{proof}
  Let $x^*=\sum_{i\geq 1} \lambda_i\1_{S_i}$ where $S_i$ are disjoint
  segments of $T$ and $\sum_{i\geq 1}\lambda_i^2\leq 1$. If there is
  an $i_0\geq 1$ such that $\lambda_{i_0}=1$, then
  $x^*=\1_{S_{i_0}}$. Since the biorthogonal functionals are
  (weak$^*$) denting points we may assume that $S_{i_0}$ contains at
  least two points of $T$. Now if we let $s_{i_0}$ be the starting
  point of this segment, and if we let
  $T_{i_0}=S_{i_0}\backslash\{s_{i_0}\}$, we have
  $\norm{x^*-e_{s_{i_0}}^*}=\norm{\1_{T_{i_0}}}=1$ so $x^*$ is at
  distance strictly less than $2$ from a (weak$^*$) denting point in
  $JT^*$ and thus cannot be a weak$^*$ Daugavet point.

  Next let us assume that $\lambda_i\in (0,1)$ for every $i\geq 1$ and
  let us fix any $i_0\geq 1$. We define as above $s_{i_0}$ and
  $T_{s_{i_0}}$ (which might eventually be empty) and we let $T_i=S_i$
  for any $i\neq i_0$. Since the $T_i$ are disjoint segments of $T$ we
  have $\norm{x^*-e_{s_{i_0}}^*}\leq
  1-\abs{\lambda_{i_0}}+\norm{x^*-\lambda_{s_{i_0}}e_{s_{i_0}}^*}
  =1-\abs{\lambda_{i_0}}+\norm{\sum_{i\geq 1}\lambda_i\1_{T_i}}\leq
  2-\lambda_{s_{i_0}}<2$ and the conclusion follows as above.
\end{proof}

In light of the above we ask:

\begin{quest}
  Can molecules in $JT^*$ be $\Delta$-points?
\end{quest}

Note that the observation from \cite[Remark 2.4]{JRZ}  applied to the set
of molecules in $JT^*$ which satisfies as mentioned $\cconv
\mathcal{M}=B_{X^*}$ gives the following simplification.

\begin{lem}
  Let $x^*\in S_{JT^*}$. Then $x^*$ is a $\Delta$-point if and only if
  we can find, for every  $\eps>0$ and for every slice $S$ of
  $B_{JT^*}$, some $m\in \mathcal{M}\cap S$ such that
  $\norm{x - m} \geq 2-\eps$.
\end{lem}

Although it is possible to do some computations in very specific cases
(for example when $x^*$ is a molecule supported on two segments) it
seems to be difficult to estimate the distance between two molecules
in general even when they are supported on the same branch of
$T$. Moreover, it is not completely trivial to distinguish norm one
molecules in $\mathcal{M}$ and to find suitable norming elements in
$JT$.

Let us mention that
with techniques similar to those for $JT,$ it is possible to prove
that also $J$ with the equivalent norms $\|\cdot\|_{\mathcal
  J}$ and $\|\cdot\|_0$ as given on p. 62 in \cite{AK}, fail to
contain $\Delta$-points.

However, if we replace the binary tree $T$ by the countably branching tree
$T_\infty=\{\emptyset\}\cup\bigcup_{n\geq 1}\Natural^n$ in the
definition of the James tree space we obtain the Banach space
$JT_\infty$ originally introduced in \cite{GM} which shares some of
the basic properties of $JT,$ but presents a few striking
dissimilarities. Indeed it is proved in \cite{GM} that the predual of
$JT_\infty$ fails the PCP and as a consequence admits no equivalent
AUC norm (in fact $\overline{\delta}_{\abs{.}}(\frac{1}{2})=0$ for any
equivalent norm $\abs{.}$ on $(JT_\infty)_*$, see \cite{Gir}) while it
satisfy the so called convex PCP, see \cite{GMS}.  For our study we
have as before that $JT_\infty$ is $2$-AUC$^*$ and admits no
Daugavet-points,  but our proof of non-existence of $\Delta$-points
fails for infinitely supported elements since restricting the support
of such element to a finite level of the tree does not necessarily
provide finitely many nodes of $T_\infty$ anymore. It is thus natural
to ask the following.

\begin{quest}
  Does $JT_\infty$ admit $\Delta$-points?
\end{quest}

\section{Almost Daugavet- and Delta-points}

Looking at Corollary~\ref{cor:unif-nsq-no-Delta}
it is natural to ask:

\begin{quest}\label{quest:delta_SR}
  Do all superreflexive Banach spaces fail to contain $\Delta$-points?
\end{quest}

This question is not trivial since, as we have seen, even
superreflexive spaces whose norm is not uniformly non-square have
small slices of the unit ball with diameter arbitrary close to $2$.
However, for a superreflexive space $X$ we know that the dual
of an ultrapower $X^{\mathcal{U}}$ is isometric to
$(X^*)^{\mathcal{U}}$, and this opens the door to using
ultrafilter limits. Since $X^{\mathcal{U}}$ is also
superreflexive we do not leave the context of superreflexive
spaces when passing to ultrapowers. Motivated by this we introduce the following
definitions which are a weakening of the notions of Daugavet- and
$\Delta$-points, and we will see how those can help simplify the
problem from Question~\ref{quest:delta_SR}.

Let $X$ be a Banach space and let $\varepsilon>0$.
We say that $x\in B_X$ is a \emph{$(2-\varepsilon)$ $Daugavet$-point}
if $B_X \subset \cconv{\Delta_\varepsilon(x)}$ or equivalently if we can
find for every $\delta>0$ and for every $x^*\in S_{X^*}$ an element
$y\in S(x^*,\delta)$ satisfying $\norm{x-y}\geq 2-\eps$.
We say that $x\in B_X$ is a \emph{$(2-\varepsilon)$ $\Delta$-point} if
$x\in \cconv{\Delta_\varepsilon(x)}$ or equivalently if we can find
for every $\delta>0$ and for every $x^*\in S_{X^*}$ such that $x\in
S(x^*,\delta)$ an element $y\in S(x^*,\delta)$ satisfying
$\norm{x-y}\geq 2-\eps$.

We say that $X$ \emph{admits almost Daugavet-points} if it admits a
  $(2-\varepsilon)$ Daugavet-point for every $\varepsilon>0$,  and we say that $X$ \emph{admits almost $\Delta$-points} if it admits a
  $(2-\varepsilon)$ $\Delta$-point for every $\varepsilon>0$. 

We say that a Banach space $X$ \emph{contains $\ell_p^n$'s uniformly}
($1 \le p \le \infty$)
if for all $\varepsilon > 0$ and $n \in \mathbb{N}$
there exist $x_1, \ldots, x_n \in B_X$
such that for all sequences of scalars $(a_k)$
\begin{equation*}
  (1+\varepsilon)^{-1}\|(a_n)\|_p
  \le
  \norm{ \sum_{k=1}^n a_k x_k}
  \le
  \|(a_n)\|_p.
\end{equation*}

We will start by studying the existence of almost Daugavet- and
$\Delta$-points in some classical spaces. We highlight the following
observation, which although obvious from the definition will provide
an easy way of constructing almost $\Delta$-points.

\begin{obs}
  Let $\varepsilon > 0$.
  If $x \in \conv \Delta_\varepsilon(x)$, then
  $x$ is a $(2 - \varepsilon)$ $\Delta$-point.
\end{obs}

Using this, we can prove the two following lemmas.

\begin{lem}
  \label{lem:ell1-finitely-rep-almost-delta-points}
  If a Banach space $X$ contains $\ell_1^n$'s uniformly,
  then $X$ admits almost $\Delta$-points.
\end{lem}

\begin{proof}
  Let $\varepsilon > 0$ and $n \in \mathbb{N}$
  and find $x_1, \ldots, x_n \in B_X$
  from the definition of containing $\ell_1^n$'s uniformly.
  Let $x = \frac{1}{n}\sum_{k=1}^n x_k$.
  Then $\|x\| \le 1$ and for each $j$
  \begin{equation*}
    \|x - x_j\|
    =
    \norm{\sum_{k \neq j} \frac{1}{n} x_k
    +
    (\frac{1}{n}-1)x_j}
    \ge
    (1+\varepsilon)^{-1}\left(2 - \frac{2}{n}\right)
  \end{equation*}
  which can be made as close to $2$ as we like.
\end{proof}

\begin{lem}
  If a Banach space $X$ contains $\ell_\infty^n$'s uniformly,
  then $X$ admits almost $\Delta$-points.
\end{lem}

\begin{proof}
  Let $\varepsilon > 0$ and $n \in \mathbb{N}$
  and find $x_1, \ldots, x_n \in B_X$
  from the definition of containing $\ell_\infty^n$'s uniformly.
  For $i = 1, \ldots, n$ we define
  \begin{equation*}
    y_i = - 2x_i + \sum_{k=1}^n x_k
  \end{equation*}
  and $x = \frac{1}{n}\sum_{i=1}^n y_i$,
  that is
  \begin{equation*}
    x = (1 - \frac{2}{n}) \sum_{k=1}^n x_k.
  \end{equation*}
  Then $\|x\| \le 1$ and for each $j$
  \begin{equation*}
    \|y_j - x\|
    =
    \norm{\sum_{k \neq j} \frac{2}{n} x_k
    +
    \left(\frac{2}{n}-2\right)x_j}
    \ge
    (1+\varepsilon)^{-1}\left(2 - \frac{2}{n}\right)
  \end{equation*}
  which can be made as close to $2$ as we like.
\end{proof}

The two lemmas above can be formulated as:
If $X$ does not have finite co-type or
does not have non-trivial type, then
$X$ admits almost $\Delta$-points.
In particular the spaces $c_0$ and $\ell_1$ both admits almost
$\Delta$-points. Let us recall that they do not admit
$\Delta$-points (by e.g. \cite[Theorem~2.17]{ALMT}).
For those spaces we can say more.

\begin{lem}
  The space $c_0$ does not admit almost Daugavet-points.
\end{lem}

\begin{proof}
  Let $x = (x_i) \in S_{c_0}$.
  Given $\delta \in (0,1)$ there exists a
  finite non-empty set $J\subset \Natural$ of cardinality $n \geq 1$
  such that $|x_j| \geq 1 - \delta$ for every $j \in J$ and
  $|x_i| < 1 - \delta$ for every $i \in \mathbb{N} \setminus J$.
  Now let
  $x^* = \frac{1}{n}\sum_{j\in J} \sgn(x_j) e_j^* \in S_{\ell_1}$.
  By Lemma~\ref{intersection_lemma} we have
  $S(x^*,\frac{\delta}{n})\subset \bigcap_{j\in J} S(\sgn(x_j) e_j^*,\delta)$
  so if $y=(y_i)_{i\geq 1}$ is in  $S(x^*,\frac{\delta}{n})$, then it
  satisfies $\abs{x_j-y_j}\leq \delta$ for every $j\in J$,  and thus
  $\norm{x - y} \leq \max\{ \delta,\ 2 - \delta\}$.
  The conclusion follows.
\end{proof}

\begin{lem}
  The space $\ell_1$ admits almost Daugavet-points.
\end{lem}

\begin{proof}
  Fix $n\geq 1$ and let $x=\frac{1}{n}\sum_{i=1}^n e_i\in S_{\ell_1}$
  where $(e_i)$ is the unit vector basis of $\ell_1$.
  Then $x$ is a  $(2-\frac{2}{n})$ $\Delta$-point by construction and
  we will show that it is in fact a $(2-\frac{2}{n})$ Daugavet-point.
  Indeed fix $x^*\in S_{\ell_\infty}$ and $\delta >0$. Since $\sup_i
  \abs{x^*(e_i)} = 1$ we can find some $i_0$ such that
  $\abs{x^*(e_{i_0})} > 1 - \delta$ that is $e_{i_0}\in S(x^*,\delta)$ or
  $-e_{i_0} \in S(x^*,\delta)$.
  Now it is easy to check that
  $\norm{x\pm e_{i_0}} \geq 2-\frac{2}{n}$ so we are done.
\end{proof}

\begin{rem}
  We have seen in Theorem~\ref{thm:prop_alpha_no_Delta} that
  if a Banach space $X$ has Rolewicz' property $(\alpha)$,
  then $X$ has no $\Delta$-points.
  But $X$ can contain almost $\Delta$-points.

  Indeed, let $T$ be the Tsirelson space.
  Even though $T$ fails Rolewicz' property $(\alpha)$
  there exists an equivalent norm $|\cdot|$, so that $(T,|\cdot|)$
  has Rolewicz' property $(\alpha)$
  by \cite[Theorems~3 and 4]{MR924764}.

  Since $T$ contains $\ell_1^n$'s uniformly
  the same holds for $(T,|\cdot|)$ by James' $\ell_1$-distortion theorem.
  Thus $(T,|\cdot|)$ admits almost $\Delta$-points by
  Lemma~\ref{lem:ell1-finitely-rep-almost-delta-points}.
\end{rem}

 The
following result, whose proof is clear from
Proposition~\ref{prop:unif_nsq_char}, covers a lot of classical norms, and in particular uniformly smooth and uniformly convex ones.

\begin{prop}
  If $X$ is uniformly non-square, then $X$ does not admit almost $\Delta$-points.
\end{prop}

It is clear from Proposition~\ref{prop:unif_nsq_char} that
the unit ball of every non uniformly non-square norm admits
small slices of diameter arbitrarily close to $2$, so ruling
out almost Daugavet- and $\Delta$-points is non-trivial
even in superreflexive spaces.

The following result is the main reason for the introduction of the
notions almost Daugavet- and almost $\Delta$-points.

\begin{prop}\label{prop:superrefl-almostDP-to-DP}
  Let $X$ be a superreflexive Banach space
  and let $\mathcal{U}$ be some free ultrafilter on $\mathbb{N}$.

  If $X$ admits almost $\Delta$- (resp. Daugavet-) points,
  then there exists $x \in X^{\mathcal{U}}$ with
  $\|x\| = 1$ such that for any slice $S$ of $B_{X^{\mathcal{U}}}$
  containing $x$ (resp. any slice $S$ of $B_{X^{\mathcal{U}}}$)
  there exists $y \in S$ with $\|y\| = 1$ and
  $\|x - y\| = 2$.

  In particular, if there exists a superreflexive Banach space
  which admits an almost $\Delta$- (resp. Daugavet-) point, then
  there exists a superreflexive Banach space
  with a $\Delta$- (resp. Daugavet-) point.
\end{prop}

\begin{proof}
  Let $\mathcal{U}$ be a free ultrafilter on $\mathbb{N}$.
  Since $X$ is superreflexive we have
  $(X^{\mathcal{U}})^* = (X^*)^{\mathcal{U}}$
  (see e.g. \cite[Proposition~7.1]{Hei1}).
  For each $n \in \mathbb{N}$  choose a $(2-\frac{1}{n})$ $\Delta$-point
  $x_n \in B_X$. Let $x = (x_n)_{\mathcal{U}}$.
  We have $\|x_n\| \ge 1 - \frac{1}{n}$ hence
  $\|x\| = \lim_{\mathcal{U}} \|x_n\| = 1$.

   Let $x^* = (x_n^*)_{\mathcal{U}}$ with $\|x^*\| = 1$
  and $\delta > 0$.
  Assume that $x \in S(x^*,\delta)$.
  This means that there is an $\eta>0$ such that
  $x^*(x) = \lim_{\mathcal{U}} x_n^*(x_n) > 1 - \delta + \eta$
  and thus there is a set $A\in \mathcal{U}$ such that
  \begin{equation*}
    x_n^*(x_n) > 1 - \delta + \eta
    = \|x_n^*\| - (\|x_n^*\| - (1 - \delta + \eta)).
  \end{equation*}
  for every $n \in A$.
  In particular we have $ \norm{x_n^*}>1-\delta+\eta$ for these $n$ and
  \begin{equation*}
    x_n \in S_n :=
    \{ y \in B_X : x_n^*(y) > \|x_n^*\| - (\|x_n^*\| - (1 - \delta + \eta))\}.
  \end{equation*}
  We can then use the definition to find $y_n \in S_n$
  with $\|x_n - y_n\| \ge 2 - \frac{1}{n}$.
  For $n \notin A$ we just let $y_n = 0$.

  Now $y = (y_n) \in X^{\mathcal{U}}$ and
  \begin{equation*}
    \|y\|
    =
    \lim_{\mathcal{U}} \|y_n\|
    \ge
    \lim_{\mathcal{U}} (\|x_n - y_n\| - \|x_n\|)
    \ge
    \lim_{\mathcal{U}} (2 - \frac{1}{n} - 1)
    = 1.
  \end{equation*}
  Finally we check that $y$ is in the slice $S(x^*,\delta)$
  \begin{equation*}
    x^*(y) = \lim_{\mathcal{U}} x_n^*(y_n)
    = \lim_{\mathcal{U}}\|x_n^*\|
    \lim_{\mathcal{U}} \frac{x_n^*}{\|x_n^*\|}(y_n)
    \geq 1 - \delta + \eta >1-\delta.
  \end{equation*}
  
  For the in particular part we just note that ultrapowers
  of superreflexive spaces are superreflexive by finite
  representability of $X^{\mathcal{U}}$ in $X$.
\end{proof}

\begin{cor}
  If $X$ is finite dimensional,
  then it does not admit almost $\Delta$-points.
\end{cor}

\begin{proof}
  This is an immediate consequence of previous proposition since
  finite dimensional spaces do not admit $\Delta$-points and since any
  ultrapower of a finite dimensional space is trivially identified
  with the space itself (under the diagonal map).
\end{proof}

By Proposition~\ref{prop:superrefl-almostDP-to-DP}
we can thus restate Question~\ref{quest:delta_SR} in the following way.

\begin{quest}
  Is it possible to find a superreflexive space admitting almost
  $\Delta$-points or almost Daugavet-points?
\end{quest}

In particular it would be interesting to investigate
the following question.

\begin{quest}
  Is it possible to find a renorming of $\ell_2$ with almost
  $\Delta$-points or almost Daugavet points?
\end{quest}

Also note that proving that superreflexive spaces do not admit
$\Delta$- (or Daugavet-) points in general would immediately imply by
the preceding proposition that superreflexive spaces do not admit
almost $\Delta$- (or Daugavet-) points.

\section*{Acknowledgments}
\label{sec:ack}

The majority of this work was done while
Y.~Perreau visited University of Agder, Kristiansand,
in October 2021, and while T.~A.~Abrahamsen and V.~Lima visited
Universit\'e Bourgogne Franche-Comt\'e, Besan\c{c}on
in November 2021.
We thank everyone who made these visits possible,
in particular, the AURORA travel programme supported by 
the Norwegian Research Council,  the French Ministry for Europe and Foreign Affairs, and the French Ministry for Higher Education, Research and Innovation.
We thank Gilles Lancien and Antonin Prochazka from the university of
Besançon for valuable discussions on the subject and for their
investment in the realization of the project. We also thank Stanimir
Troyanski for valuable discussions on the subject.

\def\polhk#1{\setbox0=\hbox{#1}{\ooalign{\hidewidth
  \lower1.5ex\hbox{`}\hidewidth\crcr\unhbox0}}} \def\cprime{$'$}
  \def\cprime{$'$} \def\cprime{$'$} \def\cprime{$'$}
\providecommand{\bysame}{\leavevmode\hbox to3em{\hrulefill}\thinspace}
\providecommand{\MR}{\relax\ifhmode\unskip\space\fi MR }
\providecommand{\MRhref}[2]{%
  \href{http://www.ams.org/mathscinet-getitem?mr=#1}{#2}
}
\providecommand{\href}[2]{#2}

\end{document}